\documentclass{elsarticle}
\usepackage{epsfig,psfrag}
\usepackage{graphicx}
\usepackage[space]{grffile}
\usepackage{amsmath}
\usepackage{mathtools}
\usepackage{amssymb}
\usepackage{color}
\usepackage{here}
\usepackage{etoolbox}
\usepackage{bbm}
\usepackage{mathrsfs}
\usepackage{multicol}
\usepackage{dsfont,accents}
\usepackage{fullpage}
\usepackage{array}
\usepackage{arydshln}
\usepackage{multirow}

\biboptions{sort&compress}

\def\llongrightarrow{\relbar\joinrel\relbar\joinrel\relbar\joinrel\rightarrow}

\providecommand{\rarrow}[1]{\stackrel{#1}{\llongrightarrow}}

\def\Xz{\boldsymbol{X}}

\providecommand\X[1]{\boldsymbol{X_{#1}}}

\everymath{\displaystyle}

\AtBeginEnvironment{bmatrix}{\setlength{\arraycolsep}{5pt}}

\newtheorem{theorem}{Theorem}

\newtheorem{proposition}[theorem]{Proposition}
\newtheorem{lemma}[theorem]{Lemma}
\newtheorem{define}[theorem]{Definition}
\newtheorem{example}[theorem]{Example}

\newtheorem{remark}[theorem]{Remark}

\numberwithin{theorem}{section}

\providecommand{\blue}[1]{\color{black}{#1}\color{black}\hspace{0pt}}
\providecommand{\bluee}[1]{\color{black}{#1}\color{black}\hspace{0pt}}

\DeclareMathOperator*{\col}{col}

\DeclareMathOperator*{\diag}{diag}

\DeclareMathOperator*{\trace}{trace}

\DeclareMathOperator{\sgn}{sgn}

\def\eps{\varepsilon}

\newenvironment{proof}{{\it Proof :~}}{\hfill$\diamondsuit$\\}

\begin{document}
\begin{frontmatter}

\title{Sign properties of Metzler matrices with applications}


\author{Corentin Briat}
\ead[url]{www.briat.info}
\ead{corentin.briat@bsse.ethz.ch, corentin@briat.info}


\address{Department of Biosystems Science and Engineering, ETH-Z\"{u}rich, Switzerland}


\begin{abstract}
Several results about sign properties of Metzler matrices are obtained. It is first established that checking the sign-stability of a Metzler sign-matrix can be either characterized in terms of the Hurwitz stability of the unit sign-matrix in the corresponding qualitative class, or in terms the negativity of the diagonal elements of the Metzler sign-matrix and the acyclicity of the associated directed graph. Similar results are obtained for the case of Metzler block-matrices and Metzler mixed-matrices, the latter being a class of Metzler matrices containing both sign- and real-type entries. The problem of assessing the sign-stability of the convex hull of a finite and summable family of Metzler matrices is also solved, and a necessary and sufficient condition for the existence of  common Lyapunov functions for all the matrices in the convex hull is obtained. The concept of sign-stability is then generalized to the concept of Ker$_+(B)$-sign-stability, a problem that arises in the analysis of certain jump Markov processes. A sufficient condition for the Ker$_+(B)$-sign-stability  of Metzler sign-matrices is obtained and formulated using inverses of sign-matrices and the concept of $L^+$-matrices. Several applications of the results are discussed in the last section.
\end{abstract}

\begin{keyword}
Sign-stability; Metzler matrices; Positive systems
\MSC[2010] 15A48, 34D05
\end{keyword}

\end{frontmatter}

\section{Introduction}

Metzler matrices \cite{Berman:94} are often encountered in various fields such as economics \cite{Nikaido:68} and biology \cite{Briat:12c,Briat:13i,Parise:14,Briat:15e}.
 These matrices have also been shown to play a fundamental role in the description of linear positive systems, a class of systems that leave the nonnegative orthant invariant \cite{Farina:00}. Due to their particular structure, such systems have been the topic of a lot of works in the dynamical systems and control communities; see e.g. also \cite{Aitrami:07,Briat:11h,Rantzer:11,Tanaka:13a,Colombino:15,Briat:15g,Briat:16c} and references therein.
In this paper, we will be interested in the sign-properties of such matrices, that is, those general properties that can be deduced from the knowledge of the sign-pattern. This problem initially emerged from economics \cite{Quirk:65}, but also found applications in  ecology \cite{May:72,May:73,Roberts:74,Jeffries:74} and chemistry \cite{Tyson:74,Clarke:75}. The rationale for this approach stems from the fact that, in these fields, the interactions between different participants in a given system are, in general, qualitatively known but quantitatively unknown. This incomplete knowledge is very often a direct consequence of the difficulty \bluee{in} identifying and discriminating models because of their inherent complexity and scarce experimental data. In this context, it seems relevant to study the properties of the system solely based on the sign pattern structure or, more loosely, from the pattern of nonzero entries. This is referred to as \emph{qualitative analysis} \cite{Maybee:69}. For instance, the sign-stability of general matrices (the property that all the matrices having the same sign-pattern are all Hurwitz stable) has been studied in \cite{Quirk:65,Jeffries:74,Jeffries:77}. An algorithm, having worst-case $O(n^2)$ time- and space-complexity, verifying the conditions in \cite{Jeffries:77} has been proposed in \cite{Klee:77}. Many other problems have also been addressed over the past decades; see e.g. \cite{Klee:84,Berger:87,Thomassen:89,Brualdi:95,Catral:09} and references therein.

The first problem addressed in this paper is the sign-stability problem for which several general necessary and sufficient conditions exist \cite{Maybee:69}. By adapting these results, we show that the sign-stability of a Metzler sign-matrix can be assessed from the Hurwitz stability of a single particular matrix, referred to as the \emph{unit sign-matrix}, lying inside the qualitative class. Therefore, for this class of matrices, checking the Hurwitz stability of an uncountably infinite and unbounded family of matrices is not more difficult than checking the stability of a given matrix. Lyapunov conditions, taking in this case the form of linear programs \cite{Boyd:04}, can hence be used for establishing the sign-stability of a given Metzler sign-matrix. An alternative condition is formulated in terms of the acyclicity of the graph associated with the sign-pattern, a property that can be easily checked using algorithms such as the Depth-First-Search algorithm \cite{Knuth:97}. This result is then generalized to the case of block matrices for which several results, potentially enabling the use of distributed algorithms, are provided. Sign-stability results are then generalized to the problem of establishing the sign-stability of all the matrices located in the convex hull of a finite number of \bluee{"summable"} Metzler sign-matrices. \blue{This problem is highly connected to the analysis of linear positive switched systems, a problem that has been extensively considered in the literature: see e.g. \cite{Gurvits:07,Mason:07,Fornasini:10}.} Necessary and sufficient conditions for the existence of a common Lyapunov function for all the matrices in the convex-hull of such matrices are also provided. \blue{A novel stability concept, referred to as Ker$_+(B)$-sign-stability, is then considered in order to formulate a constrained stability concept that arises in the analysis and the control of positive nonlinear systems and stochastic reaction networks \cite{Briat:13i,Briat:15e,Briat:16cdc}. It is shown that the obtained sufficient conditions characterizing Ker$_+(B)$-sign-stability of a Metzler sign-matrix can be brought back to a combinatorial problem known to be NP-complete \cite{Lee:98}.} Finally, mixed-matrices consisting of matrices with both sign- and real-type entries are considered and their sign-properties clarified, again in terms of algebraic and graph theoretical conditions analogous to those obtained in the pure sign-matrix case. Several application examples are provided in the last section.

\textbf{Outline:} Section \ref{sec:prel} recalls some important definitions and known results. Section \ref{sec:structmetzler} is devoted to the analysis of the sign-stability of Metzler sign-matrices. The analysis of the sign-stability of block matrices is carried out in Section \ref{sec:structinter}. Conditions for the sign-stability of the convex-hull of Metzler sign-matrices are obtained in Section \ref{sec:hull}. Section \ref{sec:relative} is concerned with relative structural stability analysis whereas Section \ref{sec:partial} addresses partial sign-stability. Application examples are finally discussed in Section \ref{sec:app}.

\textbf{Notations:} The set of real, real positive and real nonnegative numbers are denoted by $\mathbb{R}$, $\mathbb{R}_{>0}$ and $\mathbb{R}_{\ge0}$, respectively. These notations generalize to vectors where $\mathbb{R}_{>0}^n$ and $\mathbb{R}_{\ge0}^n$ are used to denote vectors with positive and nonnegative entries, respectively. \bluee{The sets of integers and nonnegative integers are denoted by $\mathbb{Z}$ and $\mathbb{Z}_{\ge0}$, respectively.} The $n$-dimensional vector of ones is denoted by $\mathds{1}_n$. The block-diagonal matrix with diagonal elements given by $M_i\in\mathbb{R}^{n_i\times m_i}$ is denoted by $\textstyle\diag_i\{M_i\}$.

\section{Preliminaries}\label{sec:prel}

We define in this section the terminology that will be used in the paper as well as some elementary results.
\begin{define}
Let us define the following matrix types and associated sets:
  \begin{itemize}
   \item A \textbf{sign-matrix} is a matrix taking entries in $\mathbb{S}:=\{\ominus,0,\oplus\}$. The set of all $n\times m$ sign-matrices is denoted by $\mathbb{S}^{n\times m}$. The set $\mathbb{S}$ has a full-order structure with the order $\ominus<0<\oplus$. \bluee{In this regard, $\ominus$ is considered here as a negative entry and $\oplus$ as a positive one.}
%
 \item \blue{A \textbf{Metzler matrix} with entries in $\mathbb{R}\cup\mathbb{S}$ is a square matrix having nonnegative off-diagonal entries. The set of all $n\times n$ Metzler matrices with entries in $\mathbb{R}\cup\mathbb{S}$ is denoted by $\mathbb{M}^n:=(\mathbb{R}\cup\mathbb{S})^{n\times n}$.}
 \blue{   \item The cones of $n\times m$ nonnegative matrices with entries in $\mathbb{S}$ or $\mathbb{R}$ are denoted by $\mathbb{S}^{n\times m}_{\ge0}$ and $\mathbb{R}^{n\times m}_{\ge0}$, respectively.}
  \blue{ \item The sets $\mathbb{MS}^n$ and $\mathbb{MR}^n$ are used to denote the sets of Metzler sign-matrices and Metzler real matrices, respectively. The considered orders are the natural one and the one defined above for sign-matrices.}
  \end{itemize}
\end{define}

\begin{define}
  The qualitative class of a matrix $A\in(\mathbb{S}\cup\mathbb{R})^{n\times m}$ is the set of matrices given by
  \begin{equation*}
  \mathcal{Q}(A):=\begin{Bmatrix}
  M\in\mathbb{R}^{n\times m} \hspace{-2mm}& \vline & \hspace{-2mm}[M]_{ij}\left\{\begin{array}{lcl}
    \hspace{-2mm}\in\mathbb{R}_{>0} & \hspace{-2mm}\textnormal{if} & \hspace{-2mm}[A]_{ij}=\oplus\hspace{-2mm}\\
    \hspace{-2mm}\in\mathbb{R}_{<0} & \hspace{-2mm}\textnormal{if} & \hspace{-2mm}[A]_{ij}=\ominus\hspace{-2mm}\\
    \hspace{-2mm}=[A]_{ij} & \multicolumn{2}{l}{${\hspace{-2mm}}\textnormal{otherwise}$\hspace{-2mm}}

  \end{array}\right\}\hspace{-1mm}
  \end{Bmatrix}.
  \end{equation*}
\end{define}

\begin{define}
  Let $M\in\mathbb{S}^{n\times m}$. Then, the unit sign-matrix $U$ associated with $M$, denoted by $U=\sgn(M)$, is defined as
  \begin{equation}
    [U]_{ij}:=\left\{\begin{array}{lcl}
      0 &\textnormal{if}& [S]_{ij}=0,\\
      1 &\textnormal{if}& [S]_{ij}=\oplus,\\
      -1 &\textnormal{if}& [S]_{ij}=\ominus
    \end{array}\right.
  \end{equation}
  where $\sgn(\cdot)$ is the extension of the signum function to sign-matrices.
\end{define}

\begin{define}[\cite{Jeffries:77}]\label{def:graphD}
    For any matrix $A\in(\mathbb{S}\cup\mathbb{R})^{n\times n}$, we define its associated directed graph as $D_A=(V,E_D)$ where ${V=\{1,\ldots,n\}}$ and
    \begin{equation*}
      E_D=\left\{(j,i):\ a_{ij}\ne0, i,j\in V,i\ne j\right\}.
    \end{equation*}
\end{define}

\begin{define}
  A matrix $M\in\mathbb{S}^{n\times n}$ is sign-stable if all the matrices $\blue{M^\prime\in\mathcal{Q}(M)}$ are Hurwitz stable.
\end{define}

\begin{define}
  A matrix $M\in\mathbb{S}^{n\times n}$ is potentially stable if $\mathcal{Q}(M)$ contains at least one matrix that is Hurwitz stable.
\end{define}

\begin{example}
Let us consider the matrices
  \begin{equation}
      M_1=\begin{bmatrix}
        \ominus & \oplus\\
         \oplus & \ominus
      \end{bmatrix}\ \textnormal{and}\ M_2=\begin{bmatrix}
        \ominus & \ominus\\
         \oplus & \ominus
      \end{bmatrix}.
  \end{equation}
  The matrix $M_1$ is not sign-stable since the determinant is not sign definite; i.e. $\det(M_1)=\oplus+\ominus$. Indeed, the entries (1,1), (1,2) and (2,2) being given, we can always find a sufficiently large entry (2,1) such that the determinant is negative and, hence, the matrix not Hurwitz stable. On the other hand, the matrix $M_2$ is sign-stable since $\trace(M_2)=\ominus$ and $\det(M_2)=\oplus$; i.e. the coefficients of the characteristic polynomial are always positive and hence all the matrices in the qualitative class are Hurwitz stable. However, the sign-matrix $\bluee{M_1}$ is potentially sign-stable as some matrices in the qualitative class are Hurwitz stable.
\end{example}

\begin{lemma}[Hurwitz stability of Metzler matrices, \cite{Berman:94}]\label{lem:Metzler_stab}
  Let us consider a matrix $A\in\mathbb{MR}^n$. Then, the following statements are equivalent:
  \begin{enumerate}[(a)]
  \item The matrix $A$ is Hurwitz stable.
  \item $A$ is nonsingular and $A^{-1}\le0$.
  \item There exists a vector $v\in\mathbb{R}^n_{>0}$ such that $v^TA<0$.
\end{enumerate}
\end{lemma}

\section{Sign-stability of Metzler sign-matrices}\label{sec:structmetzler}

The sign-stability of matrices is a problem that has been extensively studied and for which \bluee{a complete solution exists. Such a complete solution is, for instance, given in \cite[Theorem 2]{Jeffries:77} where a set of necessary and sufficient conditions is provided in terms of  algebraic and graph theoretical concepts.} We propose here to study the particular case of Metzler sign-matrices and derive a collection of conditions that are more tailored to that kind of matrices than the general ones stated in \cite{Jeffries:77}. These conditions will be shown to be simpler than those in \cite{Jeffries:77} from both a theoretical and \blue{a computational} viewpoint. It will indeed be proved that the sign-stability of a Metzler sign-matrix can be established by checking the Hurwitz stability of a single matrix in the associated qualitative class. From a computational viewpoint, we also demonstrate that the Depth-First Search (DFS) algorithm \cite{Knuth:97} can be used in order to establish the sign-stability of the matrix. This algorithm has worst-case time- and space-complexity of $O(n^2)$ and $O(n)$, respectively. 

\subsection{Preliminary results}

The following results will play a key role in the derivation of the main results of the section.
\begin{lemma}[\cite{Farina:00}]\label{lem:diag}
  Let $A\in\mathbb{MR}^n$. If $A$ is Hurwitz stable, then $A$ has negative diagonal elements.
\end{lemma}
%
%
\begin{lemma}\label{lem:maxcycle}
  Let us consider a matrix $M\in\mathbb{MS}^n$ with negative diagonal elements and assume that its associated directed graph $D_M$ is \bluee{an elementary cycle}. Then, the matrix $\sgn(M)$ is irreducible and its Perron-Frobenius eigenvalue, denoted by $\lambda_{PF}(\sgn(M))$,  is equal to 0 with $\mathds{1}_n$ as associated right-eigenvector.
\end{lemma}
\begin{proof}
Since $D_M$ is an elementary cycle, then $\sgn(M)$ can be decomposed as $\sgn(M)=P-I_n$ where $P$ is a \blue{cyclic permutation matrix} with zero entries on the diagonal. Hence, $\sgn(M)$ is irreducible since $P$ is irreducible. Furthermore, it is immediate to see that $\sgn(M)\mathds{1}=0$ and hence that 0 is an eigenvalue of $\sgn(M)$ with positive right-eigenvector equal to $\mathds{1}$. From the irreducibility property and the Perron-Frobenius theorem \cite{Berman:94}, we have that this eigenvalue is the Perron-Frobenius one and that it is unique. The proof is complete.
\end{proof}
\begin{lemma}\label{lem:nocycle}
  Let us consider a matrix $M\in\mathbb{MR}^n\cap\{-1,0,1\}^{n\times n}$. Then, the following statements are equivalent:
  \begin{enumerate}[(a)]
    \item $M$ is Hurwitz stable \blue{(in the sense that $M$ interpreted as a real matrix is Hurwitz stable)}.
    \item The diagonal elements of $M$ are negative and the directed graph $D_M$ \bluee{is acyclic}.
  \end{enumerate}
\end{lemma}
\begin{proof}
\textbf{Proof of (a) $\Rightarrow$ (b).} We use the contrapositive.  We have two cases. The first one is when at least one of the diagonal entries is nonnegative and, in this case, Lemma \ref{lem:diag} proves the implication. The second one is when there exists at least a cycle in the graph $D_M$. Let us assume that there is a single cycle of order $k\ge2$ in $D_M$ (the case of multiple cycles can be addressed in the same way). Define the matrix $\widetilde{M}\in\mathbb{MR}^n\cap\{-1,0,1\}^{n\times n}$ with negative diagonal elements and for which the graph $D_{\widetilde{M}}$ is obtained from $D_M$ by removing all the edges that are not involved in the cycle. Then, $\widetilde{M}$ is a cyclic permutation of the matrix
\begin{equation}
\begin{bmatrix}
    P-I_k & 0\\
  0 & -I_{n-k}
\end{bmatrix}
\end{equation}
where $P$ is a \blue{cyclic permutation matrix} with diagonal elements equal to 0. The left-upper block then corresponds to the nodes involved in the cycle whereas the right-lower block corresponds to those that are not. Clearly, we have that $M\ge\widetilde{M}$ (in the same permutation basis). Hence, from the Perron-Frobenius theorem and the theory of nonnegative/Metzler matrices \cite{Berman:94}, we have that $\lambda_{PF}(M)\ge\lambda_{PF}(\widetilde{M})$. Again due to the Perron-Frobenius theorem and Lemma \ref{lem:maxcycle}, we have that $\lambda_{PF}(\widetilde{M})=\lambda_{PF}(P-I_k)=0$ and hence  $\lambda_{PF}(M)\ge0$, proving then that $M$ is not Hurwitz stable.

\textbf{Proof of (b) $\Rightarrow$ (a).} Assume that statement (b) holds. Since the graph $D_M$ is acyclic, then there exists a permutation matrix $P$ such that the matrix $P^TMP$ is upper-triangular with negative elements on the diagonal and nonnegative elements in the upper-triangular part. Therefore, the matrix $M$ is Hurwitz stable since its eigenvalues coincide with its diagonal elements, which are negative. The proof is complete.
\end{proof}

\subsection{Main results}

We can now state our main result on the sign-stability of Metzler sign-matrices:
\begin{theorem}\label{th:struct}
  Let $A\in\mathbb{MS}^{n\times n}$. Then, the following statements are equivalent:
  \begin{enumerate}[(a)]
    \item\label{stat:struct:1} The matrix $A$ is sign-stable.
    \item\label{stat:struct:3} The matrix $\sgn(A)$ is Hurwitz stable.
    \item\label{stat:struct:4} There exists $v\in\mathbb{R}_{>0}^n$ such that $v^T\sgn(A)<0$.
    \item\label{stat:struct:5} The directed graph $D_A$ is acyclic and the diagonal elements of $A$ are negative.
    \item\label{stat:struct:2} There exists a permutation matrix $P$ such that $P^T AP$ is upper-triangular with negative diagonal elements.
    \item\label{stat:struct:6} The diagonal elements of $A$ are negative and the DFS algorithm applied to the graph $D_A$ returns a spanning tree with no back edge\footnote{The spanning tree is the tree generated by the DFS algorithm while searching the graph and a back edge is an edge that points from a node to one of its ancestors in the spanning tree \cite{Knuth:97}.}.
  \end{enumerate}
\end{theorem}
\blue{\begin{proof}
The proofs that \eqref{stat:struct:1} implies \eqref{stat:struct:3} and that \eqref{stat:struct:2} implies \eqref{stat:struct:1} are obvious. The proof that \eqref{stat:struct:3} is equivalent to \eqref{stat:struct:4} follows from Lemma \ref{lem:Metzler_stab}. The proof that \eqref{stat:struct:5} is equivalent to \eqref{stat:struct:3} follows from Lemma \ref{lem:nocycle}. The proof that \eqref{stat:struct:5} is equivalent to \eqref{stat:struct:2} follows from the definition of an acyclic graph and, finally, the proof that \eqref{stat:struct:5} is equivalent to \eqref{stat:struct:6} follows from the property that a graph is acyclic if and only if one of its associated spanning trees has no back edge; see  \cite{Knuth:97}.
\end{proof}}

\begin{remark}
  The condition of statement \eqref{stat:struct:3} in Theorem \ref{th:struct} states that it is enough to check the Hurwitz stability of the unique unit sign-matrix inside the qualitative class in order to establish the sign-stability of the sign-matrix and the Hurwitz stability of the entire qualitative class. It is important to stress that this condition is peculiar to Metzler matrices as demonstrated by the following example. Indeed, the matrix
\begin{equation}
  A=\begin{bmatrix}
    \ominus & \oplus & \oplus\\
    \ominus & \ominus & \ominus\\
    0 & \oplus & \ominus
  \end{bmatrix}
\end{equation}
is not sign-stable because of the presence of a cycle of order 3 in its directed graph. However, the associated unit sign-matrix given by
\begin{equation}
  \sgn(A)=\begin{bmatrix}
  -1 & 1 & 1\\
    -1 & -1 & -1\\
    0 & 1 & -1
  \end{bmatrix}
\end{equation}
is Hurwitz stable since $\sgn(A)+\sgn(A)^T$ is negative definite.
\end{remark}

\blue{\begin{remark}
The proof that \eqref{stat:struct:1} is equivalent to \eqref{stat:struct:5} in Theorem \ref{th:struct} could have been based on the general sign-stability result \cite[Theorem 2]{Jeffries:77} that involves both algebraic and graph theoretical conditions. By applying them to the particular case of Metzler sign-matrices, we obtain that a Metzler sign-matrix is sign-stable if and only if the diagonal elements are negative and its directed graph does not contain any cycle. \bluee{For general matrices, cycles of length 2 are still allowed provided that the products of the entries involved in the same cycle are negative. This is clearly not possible in the context of Metzler matrices since all the off-diagonal entries are nonnegative.}
\end{remark}}

\begin{remark}[Computational complexity]
  The statement \eqref{stat:struct:4} indicates that the Hurwitz stability of the real Metzler matrix $\sgn(A)$ can be easily checked since the problem of finding a vector $v>0$ such that $v^T\sgn(A)<0$ is a linear programming problem \cite{Boyd:04} that can be solved using modern optimization algorithms in an efficient way. By virtue of statement (f), an alternative way for checking the sign-stability of a Metzler sign-matrix is through the negativity of its diagonal elements and the acyclicity of its associated directed graph. Indeed, checking the negativity of the diagonal elements and the existence of a cycle in a given graph are both simple problems, the latter being solvable using algorithms such as the Depth-First-Search algorithm \cite{Knuth:97} which has a worst-case time-complexity of $O(n^2)$ and a worst-case space-complexity of $O(n)$. Complexity-wise, the general conditions of  \cite[Theorem 2]{Jeffries:77} can be checked using the algorithm proposed in \cite{Klee:77} which has worst-case time- and space-complexity of $O(n^2)$.
\end{remark}

For completeness, we provide several necessary conditions for the sign-stability of Metzler sign-matrices:
\begin{proposition}[Necessary conditions]
The Metzler sign-matrix $A\in\mathbb{MS}^n$ is not sign-stable if one of the following statements hold:
\begin{enumerate}[(a)]
  \item  $[A]_{ii}\ge0$ for some $i=1,\ldots,n$.
  \item $[A]_{ij}[A]_{ji}>0$ for some $i,j=1,\ldots,n$, $i\ne j$.
  \item There is a cycle in the directed graph $D_A$.
  \item The matrix $A$ is irreducible.
\end{enumerate}
\end{proposition}
\begin{proof}
  Statement (a) follows from Lemma \ref{lem:diag} while the others follow from the existence of at least one cycle in the directed graph associated with the sign-matrix.
\end{proof}

We conclude this section with a result on the potential sign-stability of Metzler sign-matrices:
\begin{theorem}[Potential sign-stability]
  Let us consider a Metzler sign-matrix $A\in\mathbb{MS}^n$. Then, the following statements are equivalent:
  \begin{enumerate}[(a)]
    \item The diagonal entries of $A$ are negative.
    \item The matrix $A$ is potentially sign-stable.
  \end{enumerate}
\end{theorem}
\begin{proof}
  \textbf{Proof of 2) $\Rightarrow$ 1)} Since the matrix $A$ is Metzler and is partially sign-stable, then there exists a matrix \blue{$A'\in\mathcal{Q}(A)$} that is Hurwitz stable. From Lemma \ref{lem:diag}, this matrix necessarily has negative diagonal elements and, hence, $A$ must have negative diagonal elements.

  \textbf{Proof of 1) $\Rightarrow$ 2)} Assume the matrix $A$ has negative diagonal entries. Let us then consider the matrix $M_{\eps}$, $\eps>0$, defined as
  \begin{equation}
    [M_\eps]_{ij}=\left\{\begin{array}{lcl}
      -1 && \textnormal{if }[A]_{ij}=\ominus,\\
      \eps && \textnormal{if }[A]_{ij}=\oplus,\\
      0 && \textnormal{if }[A]_{ij}=0
    \end{array}\right.
  \end{equation}
  and note that $M_\eps\in\mathcal{Q}(A)$. Clearly, when $\eps=0$, the matrix $M_0$ is Hurwitz stable and all the eigenvalues are equal to -1. Using the fact that the eigenvalues of $M_\eps$ are continuous with respect to $\eps$, then we can conclude that there exists an $\bar{\eps}>0$ such that for all $0\le\eps<\bar{\eps}$, the matrix $M_\eps$ is Hurwitz-stable, showing then that $A$ is potentially sign-stable. The proof is complete.
\end{proof}

It is finally interesting to note that Theorem \ref{th:struct} can also be used to obtain analogous sign-stability conditions for nonnegative sign-matrices. Note, however, that sign-stability needs to be defined here as the Schur stability of all the matrices in the qualitative class; i.e. all the matrices in the qualitative class have spectral radius less than unity. This result is stated below:
\begin{theorem}\label{th:struct2}
  Let $A\in\mathbb{S}_{\ge0}^{n\times n}$. Then, the following statements are equivalent:
  \begin{enumerate}[(a)]
    \item\label{stat:struct2:1} All the matrices in $\mathcal{Q}(A)$ have spectral radius less than unity.
    \bluee{\item\label{stat:struct2:1b} All the matrices in $\mathcal{Q}(A)$ have zero spectral radius.}
    \item\label{stat:struct2:3} The matrix $\sgn(A)-I_n$ is Hurwitz stable.
    \item\label{stat:struct2:4} There exists $v\in\mathbb{R}_{>0}^n$ such that $v^T(\sgn(A)-I_n)<0$.
    \item\label{stat:struct2:5} The directed graph $D_A$ is acyclic and the diagonal elements of $A$ are all equal to zero.
    \item\label{stat:struct2:2} There exists a permutation matrix $P$ such that $P^T AP$ is upper-triangular and has zero elements on the diagonal.
    \item\label{stat:struct2:6} There is no back edge in $D_A$ and the diagonal elements of $A$ are all equal to zero.
  \end{enumerate}
\end{theorem}
\bluee{\begin{proof}
We need to prove here that \eqref{stat:struct2:1} is equivalent to \eqref{stat:struct2:1b}. The equivalence with the other statements follows from Theorem \ref{th:struct} and the fact that a nonnegative matrix $A\in\mathbb{R}^{n\times n}_{\ge0}$ is Schur stable if and only if the Metzler matrix $A-I_n$ is Hurwitz stable. Clearly, \eqref{stat:struct2:1b} implies \eqref{stat:struct2:1}, so let us focus on the reverse implication and assume that \eqref{stat:struct2:1b} does not hold. Hence, there exists a matrix $M\in\mathcal{Q}(A)$ such that $\rho(M)=\epsilon>0$. If $\epsilon\ge1$, then \eqref{stat:struct2:1} does not hold and, on the other hand, if $\epsilon\in(0,1)$, then using the fact that $\alpha M\in\mathcal{Q}(A)$ for any $\alpha>0$,  we can conclude that, by choosing a large enough $\alpha>0$, we will have $\rho(\alpha M)>1$ and the statement \eqref{stat:struct2:1} does not hold again. This proves that \eqref{stat:struct2:1} implies \eqref{stat:struct2:1b}.
\end{proof}}

\section{Sign-stability of block matrices}\label{sec:structinter}

The objective of this section is to extend the results of the previous one to the case of Metzler block sign-matrices. At first sight, it may seem irrelevant to consider Metzler block sign-matrices as they can be analyzed using the tools derived in the previous section. The main rationale, however, is that these matrices are commonplace in fields such as control theory where they can be used to represent, for instance, interconnections of linear (positive) dynamical systems. The objective is then to find conditions that characterize the sign-stability of a given Metzler block sign-matrix  based on the sign-stability of its elementary diagonal blocks and some additional conditions capturing the interactions between the different diagonal blocks (i.e. the off-diagonal blocks). This paradigm has led to many important results such as the small-gain theorem \cite{Desoer:75a,Zhou:96} and its linear positive systems variants \cite{Briat:11h,Colombino:15}. 

First of all, several existing results about the linear $S$-procedure \cite{Yakubovich:77,Jonsson:01} are recalled and some novel ones are also obtained. The $S$-procedure is an essential tool of control theory which allows for checking the negativity of some conditions under some constraints in a computationally tractable way. 
These results are then used to obtain exact linear programming conditions characterizing the sign-stability of Metzler block real matrices.  Finally, the obtained conditions are then adapted to the analysis of the sign-stability of Metzler block-matrices. It is emphasized that such an approach enables the use of distributed algorithms for solving this problem for very large systems in a very efficient way.

\subsection{Preliminary results}

\bluee{The main theoretical tool considered here is the so-called $S$-procedure \cite{Yakubovich:77} and, more specifically, its linear version which is known to provide an exact characterization of  the nonnegativity of a linear form on a set defined by linear inequalities (e.g. a compact convex polytope).} The linear version of the $S$-procedure is recalled below for completeness.
\begin{lemma}[Linear $S$-procedure, \cite{Yakubovich:77,Jonsson:01}]\label{lem:Sproc}
  Let $\sigma_i\in\mathbb{R}^n$, $i=0,\ldots,N$, and assume that the set
  \begin{equation}
    \mathcal{S}_{\ge0}:=\left\{y\in\mathbb{R}^n:\sigma_i^Ty\ge0,\ i=1,\ldots,N\right\}
  \end{equation}
  is nonempty. Then, the following statements are equivalent:
  \begin{enumerate}[(a)]
    \item We have that $\sigma_0^Ty\ge0$ for all $y\in\mathcal{S}_{\ge0}$.
    \item There exist some multipliers $\tau_i\ge0$, $i=1,\ldots,N$, such that the inequality
\begin{equation}
        \sigma_0^Ty-\sum_{i=1}^N\tau_i\sigma_i^Ty\ge0
\end{equation}
holds for all $y\in\mathbb{R}^n$.
    \item There exists a multiplier vector $\tau\in\mathbb{R}_{\ge0}^N$ such that the inequality
\begin{equation}
        \sigma_0^Ty-\tau^T\bar{\sigma}y\ge0
\end{equation}
holds for all $y\in\mathbb{R}^n$ where $\textstyle\bar \sigma:=\col_{i=1}^N(\sigma_i^T)$.
  \end{enumerate}
\end{lemma}
\begin{proof}
  The equivalence between the first statements follows from the results in \cite{Yakubovich:77,Jonsson:01}. The last one is simply is an equivalent vector reformulation of the statement (b).
\end{proof}

\blue{With this result in mind, we can now derive an analogous result on the nonnegative orthant that involves equality instead of inequality constraints.}
\blue{Based on the linear $S$-procedure, the following result can be proved:
\begin{proposition}\label{cor:Sproc}\label{prop:Sproc}
  Let $\sigma_i\in\mathbb{R}^n$, $i=0,\ldots,N$, and assume that the set
  \begin{equation}
    \mathcal{S}_{0}^+:=\left\{y\in\mathbb{R}^n_{\ge0}:\sigma_i^Ty=0,\ i=1,\ldots,N\right\}
  \end{equation}
  is nonempty. Then, the following statements are equivalent:
  \begin{enumerate}[(a)]
    \item We have that $\sigma_0^Ty\ge0$ for all $y\in\mathcal{S}_{0}^+$.
    \item There exist some multipliers $\eta_i\in\mathbb{R}$, $i=1,\ldots,N$, such that the inequality
    \begin{equation}
      \sigma_0^Ty+\sum_{i=1}^N\eta_i\sigma_i^Ty\ge0
    \end{equation}
    holds for all $y\in\mathbb{R}^n_{\ge0}$.
    \item There exists a multiplier vector $\eta\in\mathbb{R}^N$ such that $\left(\sigma_0^T-\eta^T\bar \sigma\right)y\ge0$ for all $y\in\mathbb{R}^n_{\ge0}$ where $\textstyle\bar \sigma:=\col_{i=1}^N(\sigma_i^T)$.
    \item There exists a multiplier vector $\eta\in\mathbb{R}^N$ such that $\sigma_0^T-\eta^T\bar \sigma\ge0$ where $\textstyle\bar \sigma:=\col_{i=1}^N(\sigma_i^T)$.
  \end{enumerate}
\end{proposition}
\begin{proof}
  Defining
  \begin{equation}
    \mathcal{S}_{\le0}:=\left\{y\in\mathbb{R}^n:-\sigma_i^Ty\ge0,\ i=1,\ldots,N\right\}
  \end{equation}
  then we can see that $\mathcal{S}_0^+=\mathcal{S}_{\ge0}\cap\mathcal{S}_{\le0}\cap\mathbb{R}^n_{\ge0}$. Applying then the linear $S$-procedure, i.e. Lemma \ref{lem:Sproc}, we get that the statement (a) is equivalent to the existence of some scalars $\tau_i\ge0$, $i=1,\ldots,2N$, such that
  \begin{equation*}
    \sigma_0^Ty+\sum_{i=1}^N\left[\tau_i-\tau_{i+N}\right]\sigma_i^Ty\ge0
  \end{equation*}
  for all $y\in\mathbb{R}_{\ge0}^n$. Letting then $\eta_i:=\tau_i-\tau_{i+N}\in\mathbb{R}$ yields the equivalence between the two first statements. The equivalence between the three last statements follows from the facts that
  \begin{equation*}
    \sigma_0^Ty-\sum_{i=1}^N\eta_i\sigma_i^Ty=\left(\sigma_0^T-\eta^T\bar \sigma\right)y
  \end{equation*}
  with $\textstyle\eta=\col_{i=1}^N(\eta_i)$ and that for a given vector $c\in\mathbb{R}^n$, we have that $c^Ty\ge0$ for all $y\in\mathbb{R}^n_{\ge0}$ if and only if $c^T\ge0$.
\end{proof}}
We can now state the following important result:
\begin{theorem}\label{th:inter}
Let $A_i\in\mathbb{MR}^{n_i}$ and $B_{ij}\in\mathbb{R}_{\ge0}^{n_i\times n_{ij}}, C_{ij}\in\mathbb{R}^{n_{ij}\times n_j}_{\ge0}$, $i,j=1,\ldots,N$, $i\ne j$. Then, the following statements are equivalent:
  \begin{enumerate}[(a)]
    \item The Metzler matrix
    \begin{equation}\label{eq:jdqsjdlskdsjklsfldkj}
    \begin{bmatrix}
      A_1  & B_{12}C_{12} & \ldots & B_{1N}C_{1N}\\
      \vdots & & & \vdots\\
      B_{N1}C_{N1} & B_{N2}C_{N2} & \ldots & A_N
    \end{bmatrix}
    \end{equation}
    is Hurwitz stable.
    \item There exist vectors $v_i\in\mathbb{R}_{>0}^{n_i}$, $\ell_{ij}\in\mathbb{R}^{n_{ij}}$, $i,j=1,\ldots,N$, $i\ne j$, such that the inequalities
        \begin{equation}\label{eq:kdslkdlsqkdlsml}
\begin{array}{rcl}
      v_i^T A_i+\sum_{j\ne i}^N\ell_{ji}^T C_{ji}&<&0\\
      v_i^TB_{ij}-\ell_{ij}^T&\le&0
\end{array}
\end{equation}
hold for all $i,j=1,\ldots,N$, $j\ne i$.
\end{enumerate}
    \end{theorem}
\begin{proof}
  \blue{The equivalence between the two statements can be shown using Proposition \ref{cor:Sproc}. From Lemma \ref{lem:Metzler_stab}, a Metzler matrix $A\in\mathbb{MR}^n$ is Hurwitz stable if and only if there exists a vector $v\in\mathbb{R}_{>0}^n$ such that $v^TA<0$, which is equivalent to the existence of a vector $v\in\mathbb{R}_{>0}^n$ and a small enough scalar $\eps>0$ such that $v^TA\mathbf{x}\le-\eps v^T\mathbf{x}$ for all $\mathbf{x}\in\mathbb{R}_{\ge0}^n$. By applying this result to the matrix \eqref{eq:jdqsjdlskdsjklsfldkj}  and decomposing $\textstyle v=:\col_{i=1}^N(v_i)$, $v_i\in\mathbb{R}^{n_i}_{>0}$, and $\textstyle\mathbf{x}=:\col_{i=1}^N(\mathbf{x}_i)$, $\mathbf{x}_i\in\mathbb{R}^{n_i}_{>0}$, following the same partitioning, we get that the matrix \eqref{eq:jdqsjdlskdsjklsfldkj}  is Hurwitz stable if and only if the following feasibility problem has a solution:}
  \begin{equation*}
    \textnormal{(P)}\quad \left|\begin{array}{rcl}
      \textnormal{Find}&&v_i\in\mathbb{R}_{>0}^{n_i},\ \eps>0\\
      \textnormal{s.t.}&& \sum_{i=1}^Nv_i^T\left(A_i{\bf x}_i+\sum_{\substack{j=1\\j\ne i}}^NB_{ij}C_{ij}{\bf x}_{j}\right)\le-\eps \sum_{i=1}^Nv_i^T{\bf x}_i,\\
      \textnormal{for all }&& {\bf x}_i\in\mathbb{R}_{\ge0}^{n_i},\ i=1,\ldots,N
      \end{array}\right.
  \end{equation*}
 \blue{Letting now ${\bf w}_{ij}:=C_{ij}{\bf x}_j$, $i,j=1,\ldots,N$, $i\ne j$, we get that the inequality constraint in Problem (P) is equivalent to the fact that the inequality
  \begin{equation}\label{eq:xw}
  \begin{array}{rcl}
        \sum_{i=1}^Nv_i^T\left((A_i+\eps I){\bf x}_i+\sum_{\substack{j=1\\j\ne i}}^NB_{ij}{\bf w}_{ij}\right)&\le&0.
  \end{array}
  \end{equation}
  %
Hence, we get a problem analogous to that of Proposition \ref{cor:Sproc} with the difference that we have vector equality constraints instead of scalar ones. However, by virtue of the statement (c) or (d) of Proposition \ref{cor:Sproc}, vector of constraints can be dealt using vector multipliers (this can also be retrieved by scalarizing the constraints, considering scalar multipliers and vectorizing the result).
  %
%
  Invoking then Proposition \ref{cor:Sproc}, (b), adapted to vector equality constraints, we get that the above inequality condition holds under the vector equality constraints if and only if there exist some vector multipliers $\ell_{ij}\in\mathbb{R}^{n_{ij}}$, $i,j=1,\ldots,N$, $i\ne j$, such that
\begin{equation}
             \sum_{i=1}^Nv_i^T\left((A_i+\eps I){\bf x}_i+\sum_{\substack{j=1\\j\ne i}}^NB_{ij}{\bf w}_{ij}\right)+\sum_{\substack{i,j=1\\j\ne i}}^N\ell_{ij}^T\left(C_{ij}{\bf x}_j-{\bf w}_{ij}\right)\le0
  \end{equation}
  holds for all ${\bf x}_i\in\mathbb{R}_{\ge0}^{n_i}$ and all ${\bf w}_{ij}\in\mathbb{R}_{\ge0}^{n_{ij}}$, $i,j=1,\ldots,N$, $i\ne j$. Reorganizing the terms in the above inequality yields
  \begin{equation}\label{eq:dksodksmd}
               \sum_{i=1}^N\left(v_i^T(A_i+\eps I)+\sum_{\substack{j=1\\j\ne i}}^N\ell_{ji}^T C_{ji}\right){\bf x}_i+\sum_{\substack{i,j=1\\j\ne i}}^N\left(v_i^TB_{ij}-\ell_{ij}^T\right){\bf w}_{ij}\le0
  \end{equation}
  where the swapping of the indices in the $\ell_{ij}$ stems from the fact $\ell_{ij}$ acts on terms depending on ${\bf x}_j$ while $\ell_{ji}$ acts on terms depending on ${\bf x}_i$. Using finally the equivalence between the statements (c) and (d) of Proposition \ref{cor:Sproc}, we get that the feasibility of \eqref{eq:dksodksmd} for all ${\bf x}_i\in\mathbb{R}_{\ge0}^{n_i}$ and all ${\bf w}_{ij}\in\mathbb{R}_{\ge0}^{n_{ij}}$ is equivalent to the feasibility of
  \begin{equation}
    v_i^T(A_i+\eps I)+\sum_{\substack{j=1\\j\ne i}}^N\ell_{ji}^T C_{ji}\le0\quad\textnormal{and}\quad   v_i^TB_{ij}-\ell_{ij}^T\le0
  \end{equation}
  for all  $i,j=1,\ldots,N$, $i\ne j$, which is equivalent to the conditions \eqref{eq:kdslkdlsqkdlsml}. The proof is complete.}
\end{proof}

We can see from this result that the stability of the overall block-matrix can be broken down to the stability analysis of $N$ subproblems (one for each of the $v_i$'s) involving coupled multipliers that explicitly capture the topology of the interconnection. An important advantage of this formulation lies in its convenient form allowing for the derivation of efficient parallel algorithms for establishing the Hurwitz stability of Metzler block-matrices.

\subsection{Main results}

\bluee{Before stating the main result of the section, we define the Minkowski product between two qualitative classes associated with nonnegative sign-matrices as the set
\begin{equation}
  \mathcal{Q}(A_1)\mathcal{Q}(A_2):=\left\{M_1M_2:\ M_1\in\mathcal{Q}(A_1),M_2\in\mathcal{Q}(A_2)\right\}
\end{equation}
where $A_1\in\mathbb{S}_{\ge0}^{n_1\times m}$ and $A_2\in\mathbb{S}_{\ge0}^{m\times n_2}$.} It is also important to introduce the three following product rules for dealing with products of nonnegative sign-matrices: 1) $\oplus+\oplus=\oplus$; 2) $\oplus\cdot\oplus=\oplus$; 3) $0\cdot\oplus=0$. It seems that product of sign-matrices has never been introduced before and is defined for the first time. In addition, the Minkowski products of qualitative classes also seems to be new. A possible reason comes from the fact that the set of sign-matrices is not closed under addition or multiplication. However, the set of nonnegative sign-matrices is closed under these operations.

\bluee{\begin{remark}\label{remark:comm}
Note that because of these product rules, we have $\mathcal{Q}(A_1)\mathcal{Q}(A_2)\subset\mathcal{Q}(A_1A_2)$ and, in general, $\mathcal{Q}(A_1)\mathcal{Q}(A_2)\ne\mathcal{Q}(A_1A_2)$. Similarly, $\sgn(A_1A_2)\ne\sgn(A_1)\sgn(A_2)$ in general.
\end{remark}}

\begin{remark}\label{remark:comm}
Note that because of these product rules, we have that $\mathcal{Q}(A_1)\mathcal{Q}(A_2)\subset\mathcal{Q}(A_1A_2)$ and that, in general, $\mathcal{Q}(A_1)\mathcal{Q}(A_2)\ne\mathcal{Q}(A_1A_2)$. For instance, if
  \begin{equation*}
    A_1=\begin{bmatrix}
  \oplus \\ \oplus
\end{bmatrix}\ \textnormal{and}\ A_2=\begin{bmatrix}
  \oplus & \oplus
\end{bmatrix},
  \end{equation*}
  then $\mathcal{Q}(A_1)\mathcal{Q}(A_2)$ is the set of all $2\times 2$ rank-1 positive matrices. However, since
  \begin{equation*}
  A_1A_2=\begin{bmatrix}
  \oplus & \oplus\\
  \oplus & \oplus
\end{bmatrix},
  \end{equation*}
  then $\mathcal{Q}(A_1A_2)$ is the set of all  $2\times 2$ positive matrices. On the other hand, if
    \begin{equation*}
    A_1=\begin{bmatrix}
  \oplus \\ \oplus
\end{bmatrix}\ \textnormal{and}\ A_2=\begin{bmatrix}
  \oplus & 0
\end{bmatrix},
  \end{equation*}
  then we have that $\mathcal{Q}(A_1)\mathcal{Q}(A_2)=\mathcal{Q}(A_1A_2)$. A similar statement can be made for unit sign-matrices. We indeed have, in general, that $\sgn(A_1A_2)\ne\sgn(A_1)\sgn(A_2)$. For instance, this is the case for the matrices
  \begin{equation*}
    A_1=\begin{bmatrix}
  \oplus & \oplus\\ 0 & \oplus
\end{bmatrix}\ \textnormal{and}\ A_2=\begin{bmatrix}
  \oplus \\ \oplus
\end{bmatrix},
  \end{equation*}
  for which we have that
  \begin{equation*}
  \sgn(A_1A_2)=\begin{bmatrix}
    1\\
    1
  \end{bmatrix}\ \textnormal{and}\    \sgn(A_1)\sgn(A_2)=\begin{bmatrix}
    2\\
    1
  \end{bmatrix}.
  \end{equation*}
\end{remark}

The above remark emphasizes the difficulty arising from the possible loss of independence between the entries of a matrix resulting from the multiplication of two sign-matrices. With this warning in mind, we can state the following result:
\begin{theorem}\label{th:inter2}
Let $\tilde A_i\in\mathbb{MS}^{n_i}$ and $\tilde B_{ij}\in\mathbb{S}_{\ge0}^{n_i\times n_{ij}}, \tilde C_{ij}\in\mathbb{S}^{n_{ij}\times n_j}_{\ge0}$. Assume, moreover, that $\mathcal{Q}(\tilde B_{ij}\tilde C_{ij})=\mathcal{Q}(\tilde B_{ij})\mathcal{Q}(\tilde C_{ij})$ for all $i,j=1,\ldots,N$, $i\ne j$. Then, the following statements are equivalent:
  \begin{enumerate}[(a)]
    \item The Metzler sign-matrix
    \begin{equation}\label{eq:jdqsjdlskdsjklsfldkj4546546546}
    \tilde{A}=\begin{bmatrix}
      \tilde A_1  & \tilde B_{12}\tilde C_{12} & \ldots & \tilde B_{1N}\tilde C_{1N}\\
      \vdots & & & \vdots\\
      \tilde B_{N1}\tilde C_{N1} & \tilde B_{N2}\tilde C_{N2} & \ldots & \tilde A_N
    \end{bmatrix}
    \end{equation}
    is sign-stable.
      \item The matrix $\sgn(\tilde{A})$ is Hurwitz stable.
    \item There exist vectors $v_i\in\mathbb{R}_{>0}^{n_i}$, $\ell_{ij}\in\mathbb{R}^{n_{j}}$, $i,j=1,\ldots,N$, $j\ne i$, such that the inequalities
        \begin{equation}\label{eq:kdkdQ}
    \begin{array}{rcl}
      v_i^T \sgn(\tilde A_i)+\sum_{j\ne i}^N\ell_{ji}^T&<&0\\
      v_i^T \sgn(\tilde B_{ij}\tilde C_{ij})-\ell_{ij}^T&<&0
\end{array}
\end{equation}
for all $i,j=1,\ldots,N$, $j\ne i$.
\item The diagonal elements of $\tilde{A}$ are negative and the directed graph $D_{\tilde{A}}$ is acyclic.
\end{enumerate}
\end{theorem}
\blue{\begin{proof}
The assumption that $\mathcal{Q}(\tilde B_{ij}\tilde C_{ij})=\mathcal{Q}(\tilde B_{ij})\mathcal{Q}(\tilde C_{ij})$ for all $i,j=1,\ldots,N$, $i\ne j$, indicates that multiplications between nonnegative sign-matrices do not create qualitative classes that strictly include the initial ones. This indicates, by virtue of Theorem \ref{th:struct} and Remark \ref{remark:comm}, that the sign-stability of the matrix $\tilde A$ is equivalent to the Hurwitz stability of $\sgn(\tilde A)$, which proves the equivalence between (a) and (b). The equivalence with (d) also follows from Theorem \ref{th:struct} and Remark \ref{remark:comm}.
\
We prove now the equivalence between (b) and (c). To this aim, let us consider the matrix
\begin{equation}
    \sgn(\tilde{A})=\begin{bmatrix}
      \sgn(\tilde A_1)  & \sgn(\tilde B_{12}\tilde C_{12}) & \ldots & \sgn(\tilde B_{1N}\tilde C_{1N})\\
      \vdots & & & \vdots\\
     \sgn( \tilde B_{N1}\tilde C_{N1}) & \sgn(\tilde B_{N2}\tilde C_{N2}) & \ldots & \sgn(\tilde A_N)
    \end{bmatrix}.
\end{equation}
Note that, by virtue of Remark \ref{remark:comm}, we have that $\sgn(\tilde B_{ij}\tilde C_{ij})\ne\sgn(\tilde B_{ij})\sgn(\tilde C_{ij})$ in general. We now observe that the above matrix is of the form of the matrix \eqref{eq:jdqsjdlskdsjklsfldkj} where $A_i=\tilde A_i$, $B_{ij}=\sgn(\tilde B_{ij}\tilde C_{ij})$ and $C_{ij}=I_{n_j}$. Applying then Theorem \ref{th:inter} yields the equivalence between the statements (b) and (c).
\end{proof}}
%

\bluee{The above result admits a potentially simpler formulation from a computational viewpoint (in the case where $n_{ij}< n_j$) whenever the assumption $\sgn(\tilde B_{ij}\tilde C_{ij})=\sgn(\tilde B_{ij})\sgn(\tilde C_{ij})$ holds for all $i,j=1,\ldots,N$, $i\ne j$. In this case, we have the following result:}
\begin{theorem}\label{th:inter3}
  Let $\tilde A_i\in\mathbb{MS}^{n_i}$ and $\tilde B_{ij}\in\mathbb{S}_{\ge0}^{n_i\times n_{ij}}, \tilde C_{ij}\in\mathbb{S}^{n_{ij}\times n_j}_{\ge0}$. Assume, moreover, that $\mathcal{Q}(\tilde B_{ij}\tilde C_{ij})=\mathcal{Q}(\tilde B_{ij})\mathcal{Q}(\tilde C_{ij})$ and $\sgn(\tilde B_{ij}\tilde C_{ij})=\sgn(\tilde B_{ij})\sgn(\tilde C_{ij})$ for all $i,j=1,\ldots,N$, $i\ne j$. Then, the following statements are equivalent:
\begin{enumerate}[(a)]
  \item The matrix $\tilde{A}$ in \eqref{eq:jdqsjdlskdsjklsfldkj4546546546} is sign-stable.
  \item  There exist vectors $v_i\in\mathbb{R}_{>0}^{n_i}$, $\ell_{ij}\in\mathbb{R}^{n_{ij}}$, $i,j=1,\ldots,N$, $j\ne i$, such that the inequalities
          \begin{equation}\label{eq:kdkdQ2}
\begin{array}{rcl}
      v_i^T \sgn(\tilde A_i)+\sum_{j\ne i}^N\ell_{ji}^T\sgn(\tilde C_{ji})&<&0\\
      v_i^T \sgn(\tilde B_{ij})-\ell_{ij}^T&<&0
\end{array}
\end{equation}
for all $i,j=1,\ldots,N$, $j\ne i$.
\end{enumerate}
\end{theorem}
\blue{\begin{proof}
The proof follows the same lines as the proof of Theorem \ref{th:inter2} with the difference that we identify the matrix $\sgn(\tilde{A})$ with  the matrix \eqref{eq:jdqsjdlskdsjklsfldkj} together with $A_i=\tilde A_i$, $B_{ij}=\sgn(\tilde B_{ij})$ and $C_{ij}=\sgn(\tilde C_{ij})$, and where we have used the assumption that $\sgn(\tilde B_{ij}\tilde C_{ij})=\sgn(\tilde B_{ij})\sgn(\tilde C_{ij})$ for all $i,j=1,\ldots,N$, $i\ne j$. The rest of the proof follows from the application of Theorem \ref{th:inter}.
\end{proof}}

\blue{When the assumption that $\mathcal{Q}(\tilde B_{ij}\tilde C_{ij})=\mathcal{Q}(\tilde B_{ij})\mathcal{Q}(\tilde C_{ij})$ for all $i,j=1,\ldots,N$, $i\ne j$, is removed in the above results, then the conditions stated in the statements (b) and (c) are only sufficient for (a) to hold, but not necessary.}

\subsection{Example}

Let us consider the following matrices
\begin{equation}
  \begin{array}{lcllcllcl}
    A_1&=&\begin{bmatrix}
      \ominus & 0\\
      \oplus & \ominus
    \end{bmatrix},&B_{12}&=&\begin{bmatrix}
      \oplus\\ 0
    \end{bmatrix},&C_{21}&=&\begin{bmatrix}
    0\\  \oplus
    \end{bmatrix}^T,\\
    A_2&=&\begin{bmatrix}
      \ominus & 0\\
      0 & \ominus
    \end{bmatrix},&B_{21}&=&\begin{bmatrix}
      0\\ \oplus
    \end{bmatrix},&C_{12}&=&\begin{bmatrix}
    \oplus\\0
    \end{bmatrix}^T.
  \end{array}
\end{equation}
Clearly, the assumptions of Theorem \ref{th:inter2} are satisfied and the conditions \eqref{eq:kdkdQ2} write
\begin{equation}
  \begin{array}{rclrclrcl}
    -v_1^1+v_1^2&<&0,& -v_1^2+\ell_{21}&<&0,&v_1^1-\ell_{12}&\le&0,\\
    -v_2^1+\ell_{12}&<&0,&-v_2^2&<&0,&v_2^2-\ell_{21}&\le&0
  \end{array}
\end{equation}
where $v_i:=\col(v_i^1, v_i^2)$. These conditions are equivalent to $v_1^2<v_1^1$, $\ell_{12}\in[v_1^1,v_2^1)$ and $\ell_{21}\in[v_2^2,v_1^2)$. The conditions are readily seen to be feasible. For instance, we may pick $v_1=(5,3)$, $v_2=(7,1)$, $\ell_{12}=6$ and $\ell_{21}=2$, proving then that the block matrix $\bar{A}$ is sign-stable. If, however, the matrix $A_2$ is changed to $$A_2=\begin{bmatrix}
      \ominus & \oplus\\
      0 & \ominus
    \end{bmatrix},$$
    then $\bar{A}$ is not sign-stable anymore and the conditions \eqref{eq:kdkdQ2} are infeasible.

\section{Sign-stability of a convex-hull of Metzler matrices}\label{sec:hull}

\blue{We address in this section the problem of establishing whether, for some given family of Metzler sign-matrices $A_i\in\mathbb{MS}^n$, $i=1,\ldots,N$,  all the matrices in the convex-hull  $\mathbf{co}(A_1,\ldots,A_N)$ (when it is well-defined) are sign-stable. We first derive several intermediate results that will allow us to state under what conditions on the matrices $A_1,\ldots,A_N$, the convex-hull $\mathbf{co}(A_1,\ldots,A_N)$ exists;  i.e. $\mathbf{co}(A_1,\ldots,A_N)\subset\mathbb{MS}^n$. We will show that this problem is equivalent to the sign-summability problem of the matrices  $A_1,\ldots,A_N$. Then, we show that establishing whether  all the matrices in the convex hull of the matrices in $\mathbf{co}(A_1,\ldots,A_N)$ are sign-stable is equivalent to checking the sign-stability of $N+1$ Metzler sign-matrices, a problem having solutions that can be formulated in various ways including graph theoretical conditions or linear programs. The second problem is concerned with the existence of common quadratic/linear Lyapunov functions for a sign-summable family of Metzler sign-matrices $A_1,\ldots,A_N$. \bluee{It will be shown that a necessary and sufficient condition for the existence of such Lyapunov functions is that all the matrices in the convex-hull be sign-stable. This result can be opposed to a well-known result stating that for a family of Metzler matrices, the stability of the associated convex-hull does not necessarily imply the existence of common Lyapunov function; e.g. \cite{Fornasini:10}.}}

\begin{define}\label{def:QCS}
  Let $\mathcal{A}\subset\mathbb{MS}^{n}$ be a set of Metzler sign-matrices. The qualitative class of $\mathcal{A}$ is defined as
  \begin{equation}
    \mathcal{Q}(\mathcal{A})=\bigcup_{M\in\mathcal{A}}\mathcal{Q}(M).
  \end{equation}
\end{define}

\begin{define}
Let $A_i\in\mathbb{MS}^n$, $i=1,\ldots,N$, be given Metzler sign-matrices. We say that the convex-hull of  the matrices $A_1,\ldots,A_N$ defined as
  \begin{equation}\label{eq:coA}
  \mathbf{co}(A_1,\ldots,A_N):=\left\{\sum_{i=1}^N\alpha_i A_i:\ \alpha\ge0,\ \sum_{i=1}^N\alpha_i=1\right\}
\end{equation}
is well-defined if $\mathbf{co}(A_1,\ldots,A_N)\subset\mathbb{MS}^n$.
\end{define}

\begin{define}
  Let us consider a family of Metzler sign-matrices $A_i\in\mathbb{MS}^n$, $i=1,\ldots,N$. We say that the family is sign-summable or that the matrices are sign-summable if $\textstyle\sum_{i=1}^NA_i\in\mathbb{MS}^n$.
\end{define}

\bluee{Nonnegative matrices are trivially sign-summable as all the entries are nonnegative and, hence, the sign-summability problem did not occur in the previous section. However, the sign of the diagonal elements of Metzler matrices is indefinite and this may cause their sum to have indefinite sign. A simple necessary and sufficient condition for the sign-summability of Metzler sign-matrices is given in the result below together with a full characterization of well-defined convex-hulls of Metzler sign-matrices:
\begin{proposition}\label{prop:summable}\label{prop:hulldef}
  Let us consider a family of Metzler sign-matrices $A_i\in\mathbb{MS}^n$, $i=1,\ldots,N$. Then, the following statements are equivalent:
  \begin{enumerate}[(a)]
      \item The matrices $A_1,\ldots,A_N$ are sign-summable.
    \item For each fixed $j=1,\ldots,n$, the non-zero elements in the set $\{[A_i]_{jj}:i=1,\ldots,N\}$ have the same sign.
    \item The convex-hull $\mathbf{co}(A_1,\ldots,A_N)$ is well-defined and we have that
\begin{equation}\label{eq:basishull}
   \mathbf{co}(A_1,\ldots,A_N)=\left\{\sum_{i=1}^N\beta_i A_i:\ \beta\in\{0,1\}^N,\beta\ne0\right\}
\end{equation}
together with
\begin{equation}\label{eq:basishullQ}
   \mathcal{Q}(\mathbf{co}(A_1,\ldots,A_N))=\bigcup_{\substack{\beta\in\{0,1\}^N\\\beta\ne0}}\mathcal{Q}\left(\sum_{i=1}^N\beta_i A_i\right).
\end{equation}
  \end{enumerate}
\end{proposition}
\begin{proof}
\textbf{Proof that (a) is equivalent to (b).} First note that because of the Metzler structure of the matrices, the nonnegative off-diagonal elements can be summed freely and only the diagonal elements are problematic. In this regard, the fact that (a) implies (b) is immediate. To prove the converse, we use contraposition and assume that (a) does not hold. This implies there exists a $k\in\{1,\ldots,n\}$ for which the set $\{[A_i]_{kk}:i=1,\ldots,N\}$ contains at least one positive entry and one negative entry. Let $i_1,i_2\in\{1,\ldots,N\}$ be such that $[A_{i_1}]_{kk}=\oplus$ and $[A_{i_2}]_{kk}=\ominus$. This then implies that the sum $A_{i_1}+A_{i_2}$ does not exist since the entry $[A_{i_1}+A_{i_2}]_{kk}=\oplus+\ominus$ has indefinite sign.

\textbf{Proof that (a) is equivalent to (c).} Assume that the statement (c) holds and hence that $\mathbf{co}(A_1,\ldots,A_N)$ is well-defined or, equivalently, that all the matrices in the convex-hull are Metzler sign-matrices. Observing now that the matrix $\textstyle\sum_{i=1}^NA_i=\frac{1}{N}\sum_{i=1}^NA_i\in\mathbf{co}(A_1,\ldots,A_N)$ implies that $\textstyle\sum_{i=1}^NA_i$ is a Metzler sign-matrix, which implies in turn that the matrices $A_1,\ldots,A_N$ are sign-summable and that (a) holds. Assume now that (a) holds. To prove that the convex-hull is well-defined, first note that for any $\alpha_i>0$, we have that $\alpha_i A_i=A_i$ as $\mathcal{Q}(\alpha_i A_i)=\mathcal{Q}(A_i)$ and that for $\alpha_i=0$ we have that $\alpha_iA_i=0_{n\times n}$ since $\mathcal{Q}(\alpha_i A_i)=\{0_{n\times n}\}$. Hence, only the sign of the $\alpha_i$'s matters (i.e. whether they are  zero or positive) and, in this respect, we can substitute each $\alpha_i$ in the definition of the convex-hull by a binary parameter $\beta_i\in\{0,1\}$ where $\beta_i=1$ if $\alpha_i>0$ and 0 otherwise. This naturally leads to the alternative expression for the convex-hull given in \eqref{eq:basishull}. Using now the sign-summability property of the matrices implies that  the convex-hull is well-defined.  The expression \eqref{eq:basishullQ} for the qualitative-class of the convex-hull then immediately follows from Definition \ref{def:QCS}.
\end{proof}

Now that the sign-summability of Metzler sign-matrices has been fully characterized, we can move on to the main objective of the section, that is, the characterization of the sign-stability of all the matrices in a well-defined convex-hull of Metzler sign-matrices:}
\begin{theorem}\label{th:family}
   Let us consider a family of Metzler sign-matrices $A_i\in\mathbb{MS}^n$, $i=1,\ldots,N$. Then, the following statements are equivalent:
\begin{enumerate}[(a)]
  \item The convex-hull $\mathbf{co}(A_1,\ldots,A_N)$ is well-defined and only contains sign-stable matrices.
  \item For all $\beta\in\{0,1\}^N$, $\beta\ne0$, the matrix $A(\beta):=\textstyle\sum_{i=1}^N\beta_iA_i$ exists and is sign-stable.
  \item The matrices $A_1,\ldots,A_N$ have negative diagonal elements and the matrix $\textstyle\sum_{i=1}^NA_i$ is sign-stable.
  \item The diagonal elements of the matrices $A_1,\ldots,A_N$ are negative and the directed graph $D_{\bar{A}}$ is acyclic.
\end{enumerate}
\end{theorem}
\begin{proof}
The equivalence between the statements (a) and (b) follows from Proposition \ref{prop:hulldef} while the equivalence between the statements (c) and (d) follows from Theorem \ref{th:struct}. Let us prove then the equivalence between (b) and (c). Assume that (b) holds. Hence, for all $\beta\in\{0,1\}^N$, $\beta\ne0$, the matrix $\textstyle A(\beta)$ exists and is sign-stable. Let $\{e_i\}_{i=1}^N$ be the natural basis of $\mathbb{R}^N$, then we get that $A(e_i)=A_i$. Therefore, for any $i=1,\ldots,N$, the matrix $A_i$ is sign-stable and, hence, it has negative diagonal elements. Considering, finally, $\beta=\mathds{1}_N$, we get that $A(\mathds{1}_N)=\textstyle\sum_{i=1}^NA_i$ is sign-stable and the statement (c) holds.

\bluee{Assume now that (c) holds and let $\textstyle A(\beta):=\sum_{i=1}^K\beta_iA_i$. From the facts that the Metzler matrices $A_i$ have negative diagonal elements and $\beta\ne0$, then we can state that $\textstyle A(\beta)\le\sum_{i=1}^NA_i$ for all $\beta\in\{0,1\}^N$, $\beta\ne0$. Hence, we have that $\textstyle \sgn(A(\beta))\le\sgn(\sum_{i=1}^NA_i)$ and thus
\begin{equation}\label{eq:dskdkmlsqkdm}
  \lambda_{PF}(\sgn(A(\beta)))\le\lambda_{PF}\left(\sgn\left(\sum_{i=1}^NA_i\right)\right)\ \textnormal{for all }\beta\in\{0,1\}^N,\beta\ne0
\end{equation}
and where $\lambda_{PF}(\cdot)$ denotes the Perron-Frobenius eigenvalue. By virtue of Theorem \ref{th:struct}, we know that $\textstyle\sum_{i=1}^NA_i$ is sign stable if and only if $\textstyle\sgn\left(\sum_{i=1}^NA_i\right)$ is Hurwitz stable. As a consequence, the inequality \eqref{eq:dskdkmlsqkdm} implies that if  $\textstyle\sum_{i=1}^NA_i$ is sign stable, then $A(\beta)$ is sign stable for any $\beta\in\{0,1\}^N$, $\beta\ne0$. The proof is complete.}
\end{proof}

\subsection{Existence of common Lyapunov functions}

We first recall the following result from \cite{Liberzon:99} that we slightly adapt to our setup:
\begin{lemma}[\cite{Liberzon:99}]\label{lem:liberzon}
  Let us consider $N$ upper-triangular matrices $A_i\in\mathbb{R}^{n\times n}$, $i=1,\ldots,N$. Then, the following statements are equivalent:
  \begin{enumerate}[(a)]
    \item The matrices $A_i$, $i=1,\ldots,N$, are Hurwitz stable.
    \item There exists a symmetric positive definite matrix $Q\in\mathbb{R}^{n\times n}$ such that $A_i^TQ+QA_i$ is negative definite for all $i=1,\ldots,N$.
  \end{enumerate}
\end{lemma}

We then have the following result:
\begin{theorem}\label{th:common}
The following statements are equivalent:
  \begin{enumerate}[(a)]
    \item All the matrices in $\mathbf{co}(\mathcal{A})$ are sign-stable.
    \item For all $(A_1',\ldots,A_N')\in\mathcal{Q}(A_1)\times\ldots\times\mathcal{Q}(A_N)$, there exists a positive definite matrix $Q\in\mathbb{R}^{n\times n}$ such that $\tilde{A}^TQ+Q\tilde{A}$ is negative definite for all $\tilde{A}\in\mathbf{co}\{A_1',\ldots,A_N'\}$.
    \item For all $(A_1',\ldots,A_N')\in\mathcal{Q}(A_1)\times\ldots\times\mathcal{Q}(A_N)$, there exists a positive vector $v\in\mathbb{R}_{>0}^{n}$ such that  $v^T\tilde{A}<0$ for all $\tilde{A}\in\mathbf{co}(A')$.
  \end{enumerate}
\end{theorem}
\bluee{\begin{proof}
The implication that (b) implies (a) and (c) implies (a) are immediate. Assume that (a) holds and note that, by virtue of Theorem \ref{th:struct}, the sign-stability of the $A_i$'s and their sums is equivalent to saying that they can be simultaneously expressed as Hurwitz stable upper-triangular matrices through appropriate simultaneous row/column permutations. Assuming now that they are in such a form then the conclusion directly follows from Lemma \ref{lem:liberzon}. The proof that (a) implies (c) follows from  fact that for all upper-triangular Hurwitz stable $A_i'\in\mathcal{Q}(A_i)$, $i=1,\ldots,N$, we can recursively construct a $v\in\mathbb{R}_{>0}^n$ such that $v^TA_i'<0$ for all $i=1,\ldots,N$.
\end{proof}}

\section{Ker$_+(B)$-sign-stability of  Metzler sign-matrices}\label{sec:relative}

\subsection{Preliminaries}

In \blue{this section}, we will work with a larger class of sign-matrices, referred to as \emph{indefinite sign-matrices}, taking entries in $\mathbb{S}_\odot:=\mathbb{S}\cup\{\odot\}$ where $\odot$ denotes a sign indefinite entry that can be obtained from the rule $\oplus+\ominus=\odot$. To avoid confusion, sign-matrices will sometimes be referred to as \emph{definite sign-matrices}.

\begin{define}[Sign-qualitative class]
  The \emph{sign-qualitative class} associated with a matrix $A\in\mathbb{S}_\odot^{n\times m}$ is defined as
 \begin{equation*}
  \mathcal{SQ}(A):=\begin{Bmatrix}
  M\in\mathbb{S}^{n\times m} \hspace{-2mm}& \vline & \hspace{-2mm}[M]_{ij}\left\{\begin{array}{lcl}
    \hspace{-2mm}\in\mathbb{S}& \hspace{-2mm}\textnormal{if} & \hspace{-2mm}[A]_{ij}=\odot\hspace{-2mm}\\
    \hspace{-2mm}=[A]_{ij} & \multicolumn{2}{l}{${\hspace{-2mm}}\textnormal{otherwise}$\hspace{-2mm}}
  \end{array}\right\}\hspace{-1mm}
  \end{Bmatrix}.
\end{equation*}
\end{define}
It is important to stress that, due to the presence of sign-indefinite entries $\odot$, all the matrices in $\mathcal{SQ}(A)$ do not have the same sign-pattern. When there is no sign-indefinite entry in $A$, then $\mathcal{SQ}(A)=A$. In accordance with Definition \ref{def:QCS}, we have that
\begin{equation}
  \mathcal{Q}(\mathcal{SQ}(A)):=\bigcup_{A'\in\mathcal{SQ}(A)}\mathcal{Q}(A')
\end{equation}
which defines the set of all real matrices sharing the same sign-pattern of at least one matrix in the sign-qualitative class of the matrix $A$.

Let us define now the concept of Ker$_+(B)$-sign-stability:
\begin{define}
Let $A\in\mathbb{MS}^{n\times n}$ and $B\in\mathbb{R}^{n\times\ell}$, $\ell<n$, full-rank. We say that \textbf{$A$ is Ker$_+$($B$)-sign-stable} if for all $A^\prime\in\mathcal{Q}(A)$, there exists a $v\in\mathbb{R}_{>0}^n$ such that $v^TA^\prime<0$ and $v^TB=0$.
\end{define}
\blue{The main reason for considering the above concept is that finding such a vector is a problem which arises, for instance, in the analysis of certain stochastic Markov jump processes that can represent certain biological processes \cite{Anderson:15,Briat:13i,Briat:15e}. When $B=0$, the usual notion of sign-stability is retrieved.}

We will also need the concept of an inverse of a sign-matrix. This existence problem of inverses of sign-matrices has been well-studied in the literature; see e.g. \cite{Johnson:83,Berger:87,Thomassen:89,Eschenbach:99}. We consider here the following definition:
\blue{\begin{define}
 We say that the matrix $A\in\mathbb{MS}^{n}$ is invertible if there exists a matrix $\tilde A\in\mathbb{S}^{n\times n}$ such that for all $M\in\mathcal{Q}(A)$, there exists a matrix $N\in\mathcal{Q}(\tilde A)$ such that $MN=I$. When this is the case, we denote the inverse of $A$, $\tilde A$, by $A^{-1}$.
\end{define}
We then have the following result:
\begin{theorem}\label{th:inverse}
  Assume that the sign-matrix $A\in\mathbb{MS}^n$ is sign-stable. Then,
  \begin{enumerate}[(a)]
        \item the inverse of the sign-matrix $A$ exists, is nonpositive, and is such that $[A^{-1}]_{ii}=\ominus$, $i=1,\ldots,n$, and
            \begin{equation*}
              [A^{-1}]_{ij}=\left\{\begin{array}{lcl}
              \ominus&&\textnormal{if there is a path from node }j\textnormal{ to node }i\textnormal{ in the directed graph }D_A. \\
              0&&\textnormal{otherwise}
            \end{array}\right.
            \end{equation*}
            for all $i,j=1,\ldots,n$, $i\ne j$.
       \item we have that $\mathcal{Q}(A)^{-1}:=\left\{M^{-1}:\ M\in\mathcal{Q}(A)\right\}\subset\mathcal{Q}(A^{-1})$.
  \end{enumerate}
\end{theorem}}
\blue{\begin{proof}
To prove statement (a), we first assume, without loss of generality, that the sign-matrix $A$ is upper-triangular with negative diagonal elements. The key idea behind the proof is to demonstrate the existence of an inverse for the matrix $A$ by showing that the inverses of all the matrices in the qualitative class $\mathcal{Q}(A)$ all share the same sign-pattern. To this aim, let us introduce the following matrix
\begin{equation}
  M(m):=-\sum_{i}^nm_{i,i}e_ie_i^T+\sum_{j>i}^nm_{i,j}e_ie_j^T
\end{equation}
where $m_{i,j}\ge0$ for all $i,j=1,\ldots,N$, $j\ge i$. Note, moreover,  that we have
\begin{equation}
  \mathcal{Q}(A)=\{M(m): m_{i,j}>0\ \textnormal{if\ }(i,j)\in\mathcal{I}(A),\ m_{i,j}=0\ \textnormal{otherwise}\}
\end{equation}
where
\begin{equation}
\mathcal{I}(A):=\left\{(i,j):[A]_{ij}\ne0,\ i,j=1,\ldots,N, j\ge i\right\}.
\end{equation}
Since, by assumption, the matrix $A$ is sign-stable, then its diagonal elements are negative (and hence  $m_{i,i}>0$ for all $i=1,\ldots,N$), the directed graph $D_A$ is acyclic and all the matrices in $\mathcal{Q}(A)$ are invertible. Considering further the system of linear equations given by $y=M(m)x$ and recursively solving for $x$ yields
\begin{equation}
  \begin{array}{rcl}
    x_n&=&\dfrac{-1}{m_{n,n}}y_n\\
    x_{n-1}&=&\dfrac{1}{m_{n-1,n-1}}\left(-y_{n-1}+m_{n-1,n}x_n\right)\\
                &=&\dfrac{-1}{m_{n-1,n-1}}\left(y_{n-1}+\dfrac{m_{n-1,n}}{m_{n,n}}y_n\right)\\
    x_{n-2}&=&\dfrac{-1}{m_{n-2,n-2}}\left(y_{n-2}+\dfrac{m_{n-2,n-1}}{m_{n-1,n-1}}y_{n-1}+\left(\dfrac{m_{n-2,n-1}m_{n-1,n}}{m_{n-1,n-1}m_{n,n}}+\dfrac{m_{n-2,n}}{m_{n,n}}\right)y_n\right)\\
     & \vdots & \\
    x_{n-k}&=&\dfrac{1}{m_{n-k,n-k}}\left(-y_{n-k}+\sum_{i=n-k+1}^nm_{n-k,i}x_i\right),\ k=3,\ldots,n-1.
  \end{array}
\end{equation}
As expected the diagonal entries of $M(m)^{-1}$ are just the inverse of the diagonal entries of $M(m)$. The off-diagonal entries of $M(m)^{-1}$ are defined recursively over all possible paths in the graph $D_A$ from one node to another. For instance, the entry $[M(m)^{-1}]_{n-2,n}$ is nonzero if there is a path of length one from the node $n$ to the node $n-2$ in the directed graph $D_A$ (i.e. $m_{n-2,n}\ne0$) or a path of length two passing through the node $n-1$ (i.e. $m_{n-2,n-1}m_{n-1,n}\ne0$). This argument easily generalizes to all the entries of the inverse of the matrix  $M(m)^{-1}$ by exploiting the acyclicity of the graph $D_A$ and the nonnegativity of the off-diagonal elements of $M(m)$. As a consequence, all the nonzero entries are necessarily negative and, hence, the matrix $M(m)^{-1}$ is nonpositive. Uniqueness follows from the fact the sign-pattern of the inverse only depends on the paths in the graph $D_A$ which are, in turn, uniquely defined by the sign-matrix $A$. Consequently, an entry of $M(m)^{-1}$ is 0 if and only if there is no path from $j$ to $i$ in $D_A$, otherwise it is negative. This proves statement (a). To prove statement (b), first remark that since $A$ is a sign-stable Metzler sign-matrix, then from statement (a) we have that the matrix $A^{-1}$ is well-defined and is nonpositive with negative diagonal elements. Hence, statement (b) holds since all the inverses of the matrices in $\mathcal{Q}(A)$ have the same sign-pattern as $A^{-1}$. The proof is complete.
\end{proof}}

Two important facts are stated in the above result. First of all, the sign-pattern of the inverse sign-matrix $A^{-1}$ is uniquely defined by the sign-pattern of $A$ provided that $A$ is sign-stable. Secondly, that for any sign-stable matrix $A\in\mathbb{S}^{n\times n}$, the nonzero entries of the inverse matrix $A^{-1}$ can simply be deduced from the examination of the graph $D_A$. For comparison, it requires at most $n^2$ operations, which is better than usual real matrix inversion algorithms; see e.g. \cite{LeGall:14}. A drawback, however, is that by considering the matrix $A^{-1}$, we lose some information since the class of matrices that is considered is larger.

\begin{remark}
  It is important to stress that, for any $A\in\mathbb{MS}^n$, $A$ sign-stable, we have that $\mathcal{Q}(A)^{-1}\ne\mathcal{Q}(A^{-1})$. Indeed, the inverse of a Hurwitz stable real nonpositive matrix with negative diagonal is not necessarily Metzler, even in the upper-triangular case. In other words, the matrix inversion operation on the set of sign-stable upper-triangular nonpositive sign-matrices is not well-defined in the sense that the sign-pattern is non-uniquely defined. For instance, we have that
\begin{equation}
  M=\begin{bmatrix}
 -1   & -1&    -2  &  -4\\
     0  &  -1 &   -1 &   -5\\
     0    & 0   & -1  &  -1\\
     0    & 0    & 0&    -1
  \end{bmatrix} \textnormal{but}\ M^{-1}=\begin{bmatrix}
 -1  &   1   &  1&    -2\\
     0&    -1 &    1&     4\\
     0  &   0 &   -1  &   1\\
     0 &    0&     0  &  -1
  \end{bmatrix}
\end{equation}
is not a Metzler matrix.
\end{remark}

\blue{The following concept will be crucial for stating the main result of this section:
\begin{define}[$L^+$-matrices, \cite{Lee:98}]
  Let $A\in\mathbb{S}^{m\times n}$. We say that $A$ is an $L^+$-matrix if for all $A'\in\mathcal{Q}(A)$, the dual cone of $\{A'x:\ x\ge0\}$ defined as $\{y:y^TA'\ge0\}$ is equal to $\{0\}$.
\end{define}

Let us also recall the following result:
\begin{theorem}[\cite{Lee:98}, Theorem 2.4]\label{th:L+}
  Let $R\in\mathbb{S}^{m\times n}$. Then, the following statements are equivalent:
  \begin{enumerate}[(a)]
    \item $R$ is an $L^+$-matrix.
    \item $R$ has no zero row and for all $R'\in\mathcal{Q}(R)$, there exists a  $v\in\mathbb{R}_{>0}^n$ such that $R'v=0$.
    \item For all nonzero diagonal matrices $D\in\{-1,0,1\}^{m\times m}$, some column of $DR$ is nonzero and nonnegative.
     \item For all nonzero diagonal matrices $D\in\{-1,0,1\}^{m\times m}$, some column of $DR$ is nonzero and nonpositive.
  \item For all $R'\in\mathcal{Q}(R)$, we have that  $\{R'x:\ x\in\mathbb{R}_{\ge0}^n\}=\mathbb{R}^m$; i.e. the cone generated by $R'$ is all of $\mathbb{R}^m$.
  \end{enumerate}
\end{theorem}}

We are now ready to state the main result of this section which provides sufficient conditions for the Ker$_+$($B$)-sign-stability of a given matrix $A\in\mathbb{MS}^n$:

\begin{proposition}\label{prop:KerB}
Let $A\in\mathbb{MS}^{n\times n}$ be sign-stable and $B\in\mathbb{R}^{n\times \ell}$, $\ell<n$, be full-rank. 
Assume that one of the following equivalent statements holds:
\begin{enumerate}[(a)]
  \item All the matrices in $\mathcal{SQ}(B^TA^{-T})$ are $L^+$-matrices.
  \item All the matrices in $\mathcal{SQ}(B^TA^{-T})$ have no zero row and for all $M\in\mathcal{Q}(\mathcal{SQ}(B^TA^{-T})$, there exists a $w\in\mathbb{R}_{>0}^n$ such that $Mw=0$.
  \item For all nonzero diagonal matrices $D\in\{-1,0,1\}^{\ell\times \ell}$ and all $M\in\mathcal{SQ}(B^TA^{-T})$, some column of $DM$ is nonzero and nonnegative.
  \item For all nonzero diagonal matrices $D\in\{-1,0,1\}^{\ell\times \ell}$ and  all $M\in\mathcal{SQ}(B^TA^{-T})$, some column of $DM$ is nonzero and nonpositive.
 \item For all $M\in\mathcal{Q}(\mathcal{SQ}(B^TA^{-T})$, we have $\{Mx:\ x\in\mathbb{R}_{\ge0}^n\}=\mathbb{R}^\ell$.
\end{enumerate}
Then, the matrix $A$ is Ker$_+(B)$-sign-stable.
\end{proposition}
\blue{\begin{proof}
The equivalence between the statements of the result directly follows from Theorem \ref{th:L+}. Hence, we simply need to prove that (b), for instance, implies that the matrix $A$ is Ker$_+(B)$-sign-stable. From statement (b), we have that for all $M\in\mathcal{Q}(\mathcal{SQ}(B^TA^{-T})$, there exists a $w\in\mathbb{R}_{>0}^n$ such that $Mw=0$ and observe, moreover, that the following inclusions hold
\begin{equation}
  \mathcal{Q}(A)^{-1}B\subseteq \mathcal{Q}(A^{-1})B\subseteq \mathcal{Q}(\mathcal{SQ}(A^{-1}B)).
\end{equation}
The first inclusion follows from the statement (b) of Theorem \ref{th:inverse} while the second one is immediate from the fact that any matrix in $\mathcal{Q}(A^{-1})B$ necessarily belongs to $\mathcal{Q}(\mathcal{SQ}(A^{-1}B))$ since $\mathcal{SQ}(A^{-1}B)$ contains all the possible sign patterns that $\mathcal{Q}(A^{-1})B$ could have. Hence, if all the matrices in $\mathcal{SQ}(B^TA^{-T})$ are $L^+$-matrices, then for all $M\in B^T\mathcal{Q}(A)^{-T}$, there exist a vector $w\in\mathbb{R}_{>0}^n$ such that $Mw=0$. Decomposing $M$ as $M=B^T\tilde{A}^{-T}$, where $\tilde{A}\in\mathcal{Q}(A)$  is Hurwitz stable, yields $w^T\tilde{A}^{-1}B=0$. Letting $v^T:=-w^T\tilde{A}^{-1}$, we get that $v>0$ since $w>0$ and $\tilde{A}^{-1}\le0$. Hence, we have that $-v^TB=0$ or, equivalently, $v^TB=0$ together with $v^T\tilde A=-w^T<0$, which proves that the matrix $A$ is Ker$_+(B)$-sign-stable.
\end{proof}}

\blue{From the statements (c) and (d) of Theorem \ref{th:L+}, we can see that checking whether a matrix is an $L^+$-matrix is a combinatorial problem which turns out to be NP-complete (see \cite{Klee:84,Lee:98}). However, the approach remains applicable for matrices $R$ having a small number of rows. Applied to our problem, the number of sign-matrices on which to apply the column/row condition (see the statements  (c) and (d) of Proposition \ref{prop:KerB}) is equal to  $(3^\ell-1)p$ sign-matrices where $p$ is the cardinal of $\mathcal{SQ}(B^TA^{-T})$. Interestingly, the complexity does not directly depend on $n$, but indirectly depends on $n$ through the cardinality $p$ since a larger $n$ is likely to lead to a higher cardinality for the set $\mathcal{SQ}(B^TA^{-T})$. Finally, it is important to stress than the problem remains tractable as long as the values of $\ell$ and $p$ are sufficiently small.}

\section{Sign-stability of mixed matrices}\label{sec:partial}

As we have seen in the previous sections, sign-stability requires strong structural properties for the considered sign-matrices. The idea here is to extend the scope of sign-stability to mixed matrices which contain both sign and real entries. The idea is to enlarge the structures of the matrices (i.e. not necessarily triangular) that can be considered using such an approach. In this case, we will shown that cycles can be allowed for the nodes associated with real entries. We notably prove here that the sign-stability of a class of mixed matrices can be exactly characterized in terms of Hurwitz stability conditions, sign-stability conditions and the non-existence of cycles in a particular bipartite graph obtained from the mixed matrix.

\subsection{Preliminaries}

The following definition states an immediate extension of the sign-stability property of sign-matrices to the case of mixed matrices:
\begin{define}
  A mixed matrix $M\in\left(\mathbb{R}\cup\mathbb{S}\right)^{n\times n}$ is sign-stable if all the matrices $M^\prime\in\mathcal{Q}(M)$ are Hurwitz stable.
\end{define}

\blue{The following lemma states a result regarding the stability of Metzler block matrices which is analogous to the Schur complement formula:
\begin{lemma}\label{lem:SchurComplement}
  Let us consider a matrix $M\in\mathbb{MR}^{n_1+n_2}$ that we decompose as
  \begin{equation}
    M=:\begin{bmatrix}
      M_{11} & M_{12}\\
      M_{21} & M_{22}
    \end{bmatrix}
  \end{equation}
  where $M_{11}\in\mathbb{MR}^{n_1}$, $M_{22}\in\mathbb{MR}^{n_2}$, $M_{12}\in\mathbb{R}_{\ge0}^{n_1\times n_2}$ and $M_{21}\in\mathbb{R}_{\ge0}^{n_2\times n_1}$. Then, the following statements are equivalent:
  \begin{enumerate}[(a)]
    \item The matrix $M$ is Hurwitz stable.
    \item The matrices $M_{11}$ and $M_{22}-M_{21}M_{11}^{-1}M_{12}$ are Hurwitz stable.
    \item The matrices $M_{22}$ and $M_{11}-M_{12}M_{22}^{-1}M_{21}$ are Hurwitz stable.
  \end{enumerate}
\end{lemma}
\begin{proof}
  We only prove the equivalence between the two first statement. The proof of the equivalence between the first and third statements is analogous. To this aim, assume that $M$ is Hurwitz stable. Hence, there exist two positive vectors $\lambda_1,\lambda_2$ of appropriate dimensions such that we have
    \begin{equation}\label{kdlskdksdmskdmqkdm}
    \begin{bmatrix}
      M_{11} & M_{12}\\
      M_{21} & M_{22}
    \end{bmatrix}\begin{bmatrix}
      \lambda_1\\
      \lambda_2
    \end{bmatrix}<0.
  \end{equation}
  Since the matrices $M_{12}$ and $M_{21}$ are nonnegative, the above inequality implies that $M_{11}\lambda_1<0$ and $M_{22}\lambda_2<0$, i.e. the matrices $M_{11}$ and $M_{22}$ are both Hurwitz stable. Moreover, the inequality \eqref{kdlskdksdmskdmqkdm} is equivalent to the existence of a sufficiently small $\epsilon>0$ such that
  \begin{equation}
      M_{11}\lambda_1+M_{12}\lambda_2<0\ \textnormal{and}\ M_{21}\lambda_1+M_{22}\lambda_2\le-\epsilon\lambda_2.
  \end{equation}
  Solving for $\lambda_2$ in the second inequality yields
  \begin{equation}
    \lambda_2\ge-(M_{22}+\epsilon I_{n_2})^{-1} M_{21}\lambda_1\ge0
  \end{equation}
  where the second inequality comes from the fact that $M_{22}+\epsilon I_{n_2}$ is Metzler and Hurwitz, which implies that $(M_{22}+\epsilon I_{n_2})^{-1}\le0$. Using now the fact that $M_{12}\ge0$, we get that
    \begin{equation}
     \left(M_{11}-M_{12}(M_{22}+\epsilon I_{n_2})^{-1} M_{21}\right)\lambda_1<0.
  \end{equation}
and finally observing now that
\begin{equation}
  -M_{22}^{-1}\le-(M_{22}+\epsilon I_{n_2})^{-1}
\end{equation}
implies
    \begin{equation}
     \left(M_{11}-M_{12}M_{22}^{-1} M_{21}\right)\lambda_1<0.
  \end{equation}
  This proves the result. The proof of the converse simply consists of reversing the arguments.
\end{proof}}

The following result, which is essential in proving the main result of this section, provides a graph interpretation for the spectral radius of the product of two nonnegative matrices to be zero:
\begin{lemma}\label{lem:cyclebipartite}
  Let us consider two square nonnegative matrices $M_1\in\mathbb{R}_{\ge0}^{n\times n}$ and $M_2\in\mathbb{R}_{\ge0}^{n\times n}$. Then, the following statements are equivalent:
  \begin{enumerate}[(a)]
    \item\label{item:cyclebi:2} We have that $\rho(M_1M_2)=0$.
  \item\label{item:cyclebi:3} There is no cycle in the directed bipartite graph $B=(V_1,V_2,E)$ where ${V_i=\{v_1^i,\ldots,v_{n}^i\}}$, $i=1,2$, $E:=E_1\cup E_2$,
        \begin{equation}
          E_1:=\{(v_i^2,v_j^1)\in V_1\times V_2:\ [M_1]_{ji}\ne0\}\ \textnormal{and}\ E_2:=\{(v_i^1,v_j^2)\in V_1\times V_2:\ [M_2]_{ji}\ne0\}.
        \end{equation}
  \end{enumerate}
\end{lemma}
\bluee{\begin{proof}
Assume that \eqref{item:cyclebi:3} holds and let us define the matrix
\begin{equation}
  S:=\begin{bmatrix}
      0 & S_1\\
      S_2 & 0
  \end{bmatrix}
\end{equation}
where $S_1,S_2\in\mathbb{S}_{\ge0}^{n\times n}$ and $\sgn(S_i)=\sgn(M_i)$, $i=1,2$. Note, moreover,  that the directed graph $D_S=(V,E_S)$ (see Definition \ref{def:graphD}) coincides with the graph $B=(V_1,V_2,E)$ with $V=V_1\cup V_2$ and $E_S=E$. Therefore, the graph $B$ is acyclic if and only if $D_S$ is. From Theorem \ref{th:struct2}, this is equivalent to saying that all the matrices in $\mathcal{Q}(S)$ have zero spectral radius and,  equivalently, that they can all be expressed in upper-triangular form with zero diagonal elements modulo some permutation. Noting now that for any $X\in\mathcal{Q}(S)$ with $X_i\in\mathcal{Q}(S_i)$, we have that $\rho(X)=0$ if and only if $\rho(X_1X_2)=0$ implies that $\rho(M_1M_2)=0$ since $M_i\in\mathcal{Q}(S_i)$, $i=1,2$. This proves the implication \eqref{item:cyclebi:3} $\Rightarrow$ \eqref{item:cyclebi:2}. Assume now that \eqref{item:cyclebi:2} holds and let
\begin{equation*}
  M=\begin{bmatrix}
  0 & M_1\\
  M_2 & 0
\end{bmatrix}
\end{equation*}
from which we have that $\rho\left(M\right)=0\ \Leftrightarrow \rho(M_1M_2)=0$. This then implies that $M$ can be put in upper-triangular form modulo some permutation, which implies that the graph $D_M$ is acyclic. Noting, finally, that the graph $D_M=(V,E_M)$ coincides with $B$ yields the result.
\end{proof}}

\subsection{Main result}

Let us consider in this section mixed matrices of the form
\begin{equation}\label{eq:partstab}
  A_{\sigma\varphi}=\begin{bmatrix}
    A_{\sigma} & C_\varphi\\
    B_\varphi & A_\varphi
  \end{bmatrix}
\end{equation}
where $A_{\sigma}\in\mathbb{MS}^{n_\sigma}$, $A_\varphi \in\mathbb{MR}^{n_\varphi }$, $B_\varphi \in\mathbb{R}^{n_\varphi \times n_\sigma}_{\ge0}$ and $C_\varphi \in\mathbb{R}^{n_\sigma \times n_\varphi }_{\ge0}$.  \bluee{The following result that can be seen as the ``mixed matrix analogue" of Theorem \ref{th:struct} as it contains both algebraic and graph theoretical conditions characterizing the sign-stability of mixed matrices. While graph theoretical conditions are easier to check, algebraic ones can be combined with linear optimization problems in the same spirit as in Theorem \ref{th:struct}:}
\begin{theorem}\label{th:mixedSS}
The following statements are equivalent:
  \begin{enumerate}[(a)]
    \item\label{item:mixed:1} The mixed matrix $A_{\sigma\varphi}$ defined in \eqref{eq:partstab} is sign-stable.
    \item\label{item:mixed:2} The following statements hold:
      \begin{enumerate}[(\ref{item:mixed:2}1)]
        \item the matrix $A_{\varphi}$ is Hurwitz stable, and
        \item the matrix $M-M_\varphi$ where $M_\varphi:=C_\varphi A_\varphi^{-1}B_\varphi$ is Hurwitz stable for all $M\in\mathcal{Q}(A_\sigma)$ or, equivalently, we have that $\rho(M^{-1}M_\varphi)<1$ for all $M\in\mathcal{Q}(A_\sigma)$.
        \end{enumerate}
        \item\label{item:mixed:2b} The following statements hold:
      \begin{enumerate}[(\ref{item:mixed:2b}1)]
        \item $A_{\sigma}$ is sign-stable,
        \item $A_\varphi $ is Hurwitz stable,
        \item we have that $\rho(M_\sigma M_\varphi)=0$ where $M_\sigma:=\sgn(A_\sigma^{-1})$.
        \end{enumerate}
        \item\label{item:mixed:4} The following statements hold:
    \begin{enumerate}[(\ref{item:mixed:4}1)]
        \item $A_{\sigma}$ is sign-stable,
        \item $A_\varphi $ is Hurwitz stable,
        \item There is no cycle in the directed bipartite graph $B=(V_\sigma,V_\varphi,E)$ where $E=E_\sigma\cup E_\varphi$ with
        \begin{equation}
          E_\sigma:=\{v_i^\varphi,v_j^\sigma)\in V_\varphi\times V_\sigma:\ [M_\sigma]_{ji}\ne0\}\ \textnormal{and}\ E_\varphi:=\{(v_i^\sigma,v_j^\varphi)\in V_\sigma\times V_\varphi:\ [M_\varphi]_{ji}\ne0\}.
        \end{equation}
    \end{enumerate}
\item\label{item:mixed:3} The following statements hold:
    \begin{enumerate}[(\ref{item:mixed:3}1)]
        \item $A_{\sigma}$ is sign-stable,
        \item $A_\varphi $ is Hurwitz stable,
        \item There is no cycle in the directed graph $D_{A_{\sigma\varphi}}$ containing nodes in both $V_\sigma=\{v_1^\sigma,\ldots,v_{n_\sigma}^\sigma\}$ and $V_\varphi=\{v_1^\varphi,\ldots,v_{n_\varphi}^\varphi\}$ where $V_\sigma$ contains the first $n_\sigma$ nodes associated with $D_{A_{\sigma\varphi}}$ and $V_\varphi$  the $n_\varphi$ last ones.
    \end{enumerate}
    \item\label{item:mixed:5} The following statements hold:
      \begin{enumerate}[(\ref{item:mixed:5}1)]
        \item $\sgn(A_{\sigma})$ is Hurwitz stable,
        \item $A_\varphi $ is Hurwitz stable,
        \item The matrix $\sgn(M_\sigma M_\varphi)-I$ is Hurwitz stable.
  \end{enumerate}
  \end{enumerate}
\end{theorem}
\blue{\begin{proof}
\textbf{Proof of \eqref{item:mixed:1} $\Leftrightarrow$ \eqref{item:mixed:2}.} The matrix $A_{\sigma\varphi}$ is sign-stable if and only if for all $M\in\mathcal{Q}(A_\sigma)$, we have that the matrix
\begin{equation}
 \begin{bmatrix}
    M & C_\varphi\\
    B_\varphi & A_\varphi
  \end{bmatrix}
\end{equation}
is Hurwitz stable. Since the above matrix is Metzler, then we know from Lemma \ref{lem:SchurComplement} that it is Hurwitz stable if and only if $M$ and $M-C_\varphi A_\varphi^{-1}B_\varphi$ are both Hurwitz stable. Note also that from the fact that $A_\sigma$ is necessarily sign-stable, then there exists an $M^*\in\mathcal{Q}(A_\sigma)$ such that $M^*-M_\varphi$ is Hurwitz stable. Hence, this means that the matrix $M-C_\varphi A_\varphi^{-1}B_\varphi$ is Hurwitz stable for all $M\in\mathcal{Q}(A_\sigma)$ if and only if $\det(M-M_\varphi)\ne0$ for all $M\in\mathcal{Q}(A_\sigma)$. From the determinant formula, we get that this condition is equivalent to saying that $\det(I-M^{-1}M_\varphi)\ne0$ for all $M\in\mathcal{Q}(A_\sigma)$ and hence $\rho(M^{-1}M_\varphi)<1$ for all $M\in\mathcal{Q}(A_\sigma)$.\\

\noindent\textbf{Proof of \eqref{item:mixed:2} $\Leftrightarrow$ \eqref{item:mixed:2b}.} We prove first that \eqref{item:mixed:2} implies \eqref{item:mixed:2b}. To this aim, assume that $\rho(M^{-1}M_\varphi)<1$ for all $M\in\mathcal{Q}(A_\sigma)$ and that $A_\sigma$ is in upper-triangular form with negative diagonal elements (this last statement comes from the fact that $A_\sigma$ is sign-stable and Theorem \ref{th:struct}). Hence, from Theorem \ref{th:inverse}, its inverse is well-defined and we have that $A_\sigma^{-1}$ is nonpositive with negative diagonal elements. Since the diagonal elements of $A_\sigma^{-1}$ (which are the inverse of the diagonal elements of $A_\sigma$) can be arbitrarily large, then the spectral radius condition of statement \eqref{item:mixed:2} can only be satisfied if $M_\varphi$ is upper-triangular with zero diagonal elements, which implies that $\rho(M^{-1}M_\varphi)=0$ for all $M\in\mathcal{Q}(A_\sigma)$ and, hence, that $\rho(M_\sigma M_\varphi)=0$.

To prove the converse, let us assume that $\rho(M_\sigma M_\varphi)=0$. Using the fact that the matrices $M_\sigma$ and $ M_\varphi$ are both nonpositive and that $M_\sigma$ has negative diagonal elements, then this implies that  $\rho(M^{-1}M_\varphi)=0$ for all $M\in\mathcal{Q}(A_\sigma)$ and, hence, that $\rho(M^{-1}M_\varphi)<1$ for all $M\in\mathcal{Q}(A_\sigma)$.\\

\noindent\textbf{Proof of \eqref{item:mixed:2b} $\Leftrightarrow$ \eqref{item:mixed:4}.}  The proof of this statement directly follows from Lemma \ref{lem:cyclebipartite}.\\

\noindent\textbf{Proof of \eqref{item:mixed:4} $\Leftrightarrow$ \eqref{item:mixed:3}.} The equivalence follows from the fact that the cycles in $D_{A_{\sigma\varphi}}$ that contain nodes in both $V_\sigma$ and $V_\varphi$ are exactly those in the graph $B$.\\

\noindent\textbf{Proof of \eqref{item:mixed:2b} $\Leftrightarrow$ \eqref{item:mixed:5}.} Clearly the statements (\ref{item:mixed:2b}1) and (\ref{item:mixed:2b}2) are equivalent to (\ref{item:mixed:5}1) and (\ref{item:mixed:5}2), respectively. The condition that $\rho(M_\sigma M_\varphi)=0$ has been shown to be equivalent to saying that $M_\sigma M_\varphi$ is upper-triangular modulo some cyclic permutation and that its diagonal elements are all equal to zero. Assuming then that $\rho(M_\sigma M_\varphi)=0$, this implies that the matrix $\sgn(M_\sigma M_\varphi)$ has diagonal elements equal to zero and is upper-triangular, which implies, in turn, that the Metzler matrix $\sgn(M_\sigma M_\varphi)-I$ is Hurwitz stable. To prove the converse, let us assume that the Metzler matrix $\sgn(M_\sigma M_\varphi)-I$ is Hurwitz stable. This, then implies, from Theorem \ref{th:struct} that it must be upper-triangular with negative elements on the diagonal, which implies then that $\sgn(M_\sigma M_\varphi)$ is triangular with diagonal elements equal to zero. This proves the result.
\end{proof}}

\blue{\begin{remark}[Linear program condition]
  Interestingly, the conditions of statement (4) of Theorem \ref{th:mixedSS} can be exactly reformulated as the following linear feasibility problem: Find $v,z\in\mathbb{R}_{>0}^{n_\sigma}$ and $w\in\mathbb{R}_{>0}^{n_\varphi}$ such that $v^T\sgn(A_\sigma)<0$,  $w^TA_\varphi<0$  and  $z^T(\sgn(M_\varphi M_\sigma)-I)<0$.
\end{remark}}

\subsection{Example}

Let us consider the matrix
\begin{equation}\label{eq:partiallyatsbale}
  A_{\sigma\varphi}=\begin{bmatrix}
    \ominus & \oplus & \vline & 1 & 0\\
    0 & \ominus & \vline & 0 & 0\\
    \hline
    \multicolumn{2}{c}{\multirow{2}{*}{$B_\varphi$}} & \vline & -1 & 2\\
     && \vline & 1 & -5
  \end{bmatrix}.
\end{equation}
Clearly, $A_\sigma$ is sign-stable and $A_\varphi$ is Hurwitz stable. Assume now that $B_\varphi$ is given by
\begin{equation}
  B_\varphi=\begin{bmatrix}
    0 & 0\\
    0 & 1
  \end{bmatrix}.
\end{equation}
In this case, the matrices $M_\sigma$ and $M_\varphi$ are given by
\begin{equation}
  M_\sigma=-\begin{bmatrix}
    1 & 1\\
    0 & 1
  \end{bmatrix}\ \textnormal{and}\ M_\varphi=-\dfrac{2}{3}\begin{bmatrix}0 & 1\\0 & 0
  \end{bmatrix}
\end{equation}
and we have that
\begin{equation}
  M_\sigma M_\varphi=\dfrac{2}{3}\begin{bmatrix}
    0 & 1\\
    0 & 0
  \end{bmatrix}
\end{equation}
which has clearly zero spectral radius. Therefore, the matrix $A_{\sigma\varphi}$ in \eqref{eq:partiallyatsbale} is sign-stable. The same conclusion can be drawn by checking the no-cycle condition in the graphs.\\

\noindent Now, if we let
\begin{equation*}
  B_\varphi=\begin{bmatrix}
    0 & 0\\
    1  & 0
  \end{bmatrix}
\end{equation*}
then we have that
\begin{equation*}
M_\varphi=-\dfrac{1}{3}\begin{bmatrix}0 & 0\\1 & 0
  \end{bmatrix}\ \textnormal{and}\  M_\sigma M_\varphi =\dfrac{1}{3}\begin{bmatrix}
    0 & 1\\
    0 & 1
  \end{bmatrix}.
\end{equation*}
Clearly, we have that $\rho( M_\sigma M_\varphi)=1/3$ and, hence, the matrix $A_{\sigma\varphi}$ in \eqref{eq:partiallyatsbale} is not sign-stable.

\section{Applications}\label{sec:app}

The goal of this section is to apply some of the obtained results to some problems arising in the structural analysis of linear positive time-delay systems (Section \ref{sec:delay} and Section \ref{sec:delay_DT}), linear positive switched systems (Section \ref{sec:switched}) and linear positive impulsive systems (Section \ref{sec:impulsive}). The problem of establishing the structural attractiveness and forward invariance of a compact set for a class of nonlinear positive dynamical systems is considered in Section \ref{sec:nlps}. Finally, the problem of establishing the structural ergodicity of a class of continuous-time Markov jump process arising in the analysis of biochemical reaction networks is addressed in Section \ref{sec:ergo}.

\subsection{Linear positive dynamical systems}

\subsubsection{Continuous-time systems with discrete delays}\label{sec:delay}

Linear positive systems with discrete-delays arise in problems such as in power-control in networks \cite{Forschini:02,Zappavigna:12} and have been theoretically studied in several papers; see e.g. \cite{Haddad:04,Briat:11h}. Let us consider the following linear system with constant discrete-delays \cite{Niculescu:01,GuKC:03,Briat:book1}:
\begin{equation}\label{eq:ltipd}
  \begin{array}{lcl}
    \dot{x}(t)&=&A_0x(t)+\sum_{i=1}^NA_ix(t-h_i)\\
       x(s)&=&\phi(s),\ s\in[-\bar{h},0]
  \end{array}
\end{equation}
where $A_0,A_i\in\mathbb{R}^{n\times n}$, $h_i>0$, $i=1,\ldots,N$, $x\in\mathbb{R}^n$, $\phi\in C([-\bar{h},0],\mathbb{R}^n)$ and $\textstyle\bar{h}:=\max_i\{h_i\}$. \bluee{It is known \cite{Haddad:04} that the above system is positive for all delays $h_i\ge0$ if and only if (a) the matrix $A_0$ is Metzler and (b) the matrices $A_i$, $i=1,\ldots,N$, are all nonnegative; and that it is asymptotically stable for any delays $h_i\ge0$, $i=1,\ldots,N$, if and only if the matrix $\textstyle\sum_{i=0}^NA_i$ is Hurwitz stable.}
\begin{remark}[Extension to time-varying delays]
By virtue of the results in \cite{AitRami:09,Briat:11h,Shen:15}, the above result remains valid when time-varying delays are considered instead of constant ones.
\end{remark}
Before stating the main result, it seems important to recall that the sign-summability of Metzler and nonnegative matrices has been fully characterized in Proposition \ref{prop:summable} (note that nonnegative matrices form a subset of Metzler matrices). With this in mind, we can state the following result:
\begin{proposition}\label{prop:CTD}
  Let $M_0\in\mathbb{MS}^{n}$ and $M_i\in\mathbb{S}_{\ge0}^{n}$, $i=1,\ldots,N$. Then, the following statements are equivalent:
  \begin{enumerate}[(a)]
  \item For all $A_0\in\mathcal{Q}(M_0)$ and all $A_i\in\mathcal{Q}(M_i)$, $i=1,\ldots,N$, the linear positive system with delays \eqref{eq:ltipd} is asymptotically stable; i.e. the linear positive system with discrete-delays \eqref{eq:ltipd} is structurally stable.
  \item For all $A_0\in\mathcal{Q}(M_0)$ and all $A_i\in\mathcal{Q}(M_i)$, $i=1,\ldots,N$, there exists a vector $v\in\mathbb{R}^n_{>0}$ such that $v^T(\textstyle\sum_{i=0}^NA_i)<0$.
  \item The matrices $M_0,\ldots,M_N$ are sign-summable and the matrix $\textstyle\sum_{i=0}^NM_i$ is sign-stable.
   \item  The matrices $M_0,\ldots,M_N$ are sign-summable and there exists a permutation matrix $P$ such that the matrix $P^T(\textstyle\sum_{i=0}^NM_i)P$ is upper-triangular.
   \item There exists a vector $v\in\mathbb{R}_{>0}^n$ such that $\textstyle v^T\sgn\left(\textstyle\sum_{i=0}^NM_i\right)<0$ holds.
\end{enumerate}
\end{proposition}
\begin{proof}
The equivalence between the statements (b), (c), (d) and (e) follows from Theorem \ref{th:struct} while the proof that (a) and (b) are equivalent follows from the stability result for linear positive systems with delays.
\end{proof}

\blue{\subsubsection{Discrete-time systems with discrete delays}\label{sec:delay_DT}

Let us now consider the case of linear discrete positive systems \cite{Buslowicz:08}. To this aim, we consider the following system \cite{Verriest:95b}
\begin{equation}\label{eq:ltipd_DT}
  \begin{array}{rcl}
    x(k+1)&=&A_0x(k)+\sum_{i=1}^NA_ix(k-h_i)\\
       x(s)&=&\phi_s,\ s\in\{-\bar{h},\ldots,0\}
  \end{array}
\end{equation}
where $A_0,A_i\in\mathbb{R}^{n\times n}$, $h_i\in\mathbb{Z}_{\ge0}$, $i=1,\ldots,N$, $x\in\mathbb{R}^n$, $\phi_s\in\mathbb{R}^n$, $s=-\bar{h},\ldots,0$ and $\textstyle \bar{h}:=\max_i\{h_i\}$. \bluee{It is known \cite{Buslowicz:08} that the above system is positive for all for all delays $h_i\in\mathbb{Z}_{\ge0}$ if and only if the matrices $A_i$, $i=0,\ldots,N$, are all nonnegative; and that it is asymptotically stable for any delays $h_i\ge0$, $i=1,\ldots,N$ if and only if the matrix $\textstyle\sum_{i=0}^NA_i$ is Schur stable.} With these results in mind, we can state the following result:
\begin{proposition}
  Let $M_i\in\mathbb{S}_{\ge0}^{n}$, $i=0,\ldots,N$. Then, the following statements are equivalent:
  \begin{enumerate}[(a)]
  \item For all $A_i\in\mathcal{Q}(M_i)$, $i=0,\ldots,N$, the linear positive system with delays \eqref{eq:ltipd_DT} is asymptotically stable; i.e. the linear positive system with discrete-delays \eqref{eq:ltipd_DT} is structurally stable.
  \item For all $A_i\in\mathcal{Q}(M_i)$, $i=0,\ldots,N$, there exists a vector $v\in\mathbb{R}^n_{>0}$ such that $v^T(\textstyle\sum_{i=0}^NA_i-I)<0$.
  \item The matrix $\textstyle\sum_{i=0}^NM_i$ is sign-stable.
   \item  There exists a permutation matrix $P$ such that the matrix $P^T(\textstyle\sum_{i=0}^NM_i)P$ is upper-triangular.
   \item There exists a vector $v\in\mathbb{R}_{>0}^n$ such that $\textstyle v^T\left[\sgn\left(\textstyle\sum_{i=0}^NM_i\right)-I\right]<0$ holds.
\end{enumerate}
\end{proposition}
\begin{proof}
The proof follows the same lines as the proof of Proposition \ref{prop:CTD} with the difference that Theorem \ref{th:struct2} is used instead of Theorem \ref{th:struct}.
\end{proof}}

\subsubsection{Continuous-time switched systems}\label{sec:switched}

Linear positive switched systems have been widely studied in the literature \cite{Gurvits:07,Mason:07,Fornasini:10} as they can represent a wide variety of real-world processes; see e.g. \cite{Ogura:16a,Ogura:16b}. Let us consider, as a starting point, the following linear switched system \cite{Liberzon:03}:
\begin{equation}\label{eq:switched}
  \begin{array}{lcl}
    \dot{x}(t)&=&A_{\sigma(t)}x(t)\\
    x(0)&=&x_0
  \end{array}
\end{equation}
where $x\in\mathbb{R}^n$ and $\sigma:\mathbb{R}_{\ge0}\to\{1,\ldots,N\}$ is a piecewise constant switching signal and $A_i\in\mathbb{R}^{n\times n}$. It is known that \cite{Gurvits:07,Mason:07} the above linear switched system is positive if and only if $A_i\in\mathbb{MR}^n$ for all $i=1,\ldots,N$.

\begin{proposition}
Let $M_i\in\mathbb{MS}^n$, $i=1,\ldots,N$. Then, the following statements are equivalent:
\begin{enumerate}[(a)]
  \item\label{item:switched:1} For all $A_i\in\mathcal{Q}(M_i)$, $i=1,\ldots,N$, the linear switched system \eqref{eq:switched} is asymptotically stable under arbitrary switching i.e. the system \eqref{eq:switched} is structurally stable under arbitrary switching.
  \item\label{item:switched:2} For all $A_i\in\mathcal{Q}(M_i)$, $i=1,\ldots,N$, the linear time-varying positive system
  \begin{equation}
    \dot{x}(t)=A(t)x(t)
  \end{equation}
  where $A(t)\in\mathbf{co}(A_1,\ldots,A_N)$, is robustly asymptotically stable.
    \item \label{item:switched:3} The diagonal elements of the matrices $M_i$, $i=1,\ldots,N$, are negative and the matrix $\textstyle\sum_{i=1}^NM_i$ is sign-stable.
  \item\label{item:switched:4} For all $A_i\in\mathcal{Q}(M_i)$, $i=1,\ldots,N$, there exists a positive definite matrix $Q\in\mathbb{R}^{n\times n}$ such that $\tilde{A}^TQ+Q\tilde{A}$ is negative definite for all $\tilde{A}\in\mathbf{co}(A_1,\ldots,A_N)$.
      \item\label{item:switched:5} For all $A_i\in\mathcal{Q}(M_i)$, $i=1,\ldots,N$, there exists a vector $v\in\mathbb{R}_{>0}^n$ such that $v^T\tilde{A}<0$ for all $\tilde{A}\in\mathbf{co}(A_1,\ldots,A_N)$.
\end{enumerate}
\end{proposition}
\begin{proof}
  The proof of the equivalence between the statements \eqref{item:switched:1} and \eqref{item:switched:2}  follows from \cite[Theorem 3]{Lin:09} while the equivalence between the statements  \eqref{item:switched:3},  \eqref{item:switched:4}    and  \eqref{item:switched:5}   follows from Theorem \ref{th:family} and Theorem \ref{th:common}. Clearly, statement  \eqref{item:switched:4}   implies statement \eqref{item:switched:2}  using standard results on switched systems; see e.g. \cite{Liberzon:03}. Finally, statement \eqref{item:switched:2}  implies  \eqref{item:switched:3}   since if the conditions of statement  \eqref{item:switched:3}   does not hold, then there will exist some matrices $A_i\in\mathcal{Q}(M_i)$, $i=1,\ldots,N$,  for which the set $\mathbf{co}(A_1,\ldots,A_N)$ will contain unstable matrices. The proof is completed.
\end{proof}

\subsubsection{Continuous-time impulsive systems}\label{sec:impulsive}

Linear positive impulsive systems have been less studied than positive switched system but some results can be found in \cite{Briat:16c}. Let us consider, as a starting point, the following linear impulsive system:
\begin{equation}\label{eq:impulsive}
  \begin{array}{rcl}
    \dot{x}(t)&=&Ax(t),\ t\ne t_k\\
    x(t_k^+)&=&Jx(t_k),\ k=1,\ldots\\
    x(t_0)&=&x_0
  \end{array}
\end{equation}
where $x\in\mathbb{R}^n$ is the state of the system. The sequence of jumping instants $\{t_k\}_{k=0}^\infty$ is assumed to be increasing and to grow unbounded. It is known \cite{Briat:16c} that this system is positive if and only if $A\in\mathbb{MR}^n$ and $J\in\mathbb{R}^{n\times n}_{\ge0}$.

\begin{proposition}
  Let $M_A\in\mathbb{MS}^n$ and $M_J\in\mathbb{S}_{\ge0}^{n\times n}$. Then, the following statements are equivalent:
  \begin{enumerate}[(a)]
    \item For all $A\in\mathcal{Q}(M_A)$ and all $J\in\mathcal{Q}(J_A)$, there exists a vector $\lambda\in\mathbb{R}^n_{>0}$ such that
    \begin{equation}\label{eq:dksldkmsqdmksmdimpulsive}
      \lambda^TA<0\ \textnormal{and}\ \lambda^T(J-I)<0.
    \end{equation}
    \item $M_A$ has negative diagonal, $M_J$ has zero diagonal and the matrix $M_A+M_J$ is sign-stable.
    %
    %
  \end{enumerate}
  Moreover, when one of the above equivalent statements holds, then for all $A\in\mathcal{Q}(M_A)$ and all $J\in\mathcal{Q}(J_A)$, the linear positive impulsive system is asymptotically stable for any increasing sequence $\{t_k\}_{k=0}^\infty$ such that $t_k\to\infty$ as $k\to\infty$.
\end{proposition}
\bluee{\begin{proof}
The proof that the first statement implies the stability under arbitrary jump sequence can be found in \cite{Briat:16c} and relies on the use of the linear copositive Lyapunov function $V(x)=\lambda^Tx$, $\lambda\in\mathbb{R}^n_{>0}$. Assume that (b) holds, then $M_A$ is sign-stable and $M_A+M_J$ is upper-triangular modulo some simultaneous row/column permutations which both imply that $M_J$ has zero diagonal entries or, equivalently, that for all  $J\in\mathcal{Q}(M_J)$, the matrix $J-I_n$ is Metzler and Hurwitz stable. Since these matrices are upper-triangular, we can finally use Theorem \ref{th:common} to prove that (a) holds. Assume now that (a) holds. Hence, we have that $M_A$ and  $M_J-I_n$ are sign-stable and, hence, $M_A$ and $M_J$ has negative and zero diagonal, respectively. Hence, $M_A$ and $M_J$ are sign-summable and we have that $\lambda^T(A+J-I_n)<0$ for all $A\in\mathcal{Q}(M_A),J\in\mathcal{Q}(M_J)$. This implies that $M_A+M_J-I_n$ is sign-stable or, equivalently, that it is upper-triangular modulo some simultaneous row/column permutations, which is equivalent to saying that $M_A+M_J$ is sign-stable.
\end{proof}}

\subsection{Nonlinear positive systems}\label{sec:nlps}

Let us consider the following class of nonlinear positive systems that arises, among others, in reaction network theory
\begin{equation}\label{eq:nlp}
\begin{array}{lcl}
    \dot{x}(t)&=&Ax(t)+Bf(x(t))+b\\
    x(0)&=&x_0
\end{array}
\end{equation}
where $x\in\mathbb{R}_{\ge0}^n$, $f:\mathbb{R}_{\ge0}^n\to\mathbb{R}_{\ge0}^\ell$, $b\in\mathbb{R}^n_{\ge0}$, $A\in\mathbb{R}^{n\times n}$ and $B\in\mathbb{Z}^{n\times\ell}$.
\begin{proposition}
  Let $A\in\mathbb{MR}^{n}$ and assume that there exist a $v\in\mathbb{R}_{>0}^n$ and an $\eps>0$ such that $v^TA\le-\eps v^T$ and $v^TB=0$. Then, for all $x_0\in\mathbb{R}_{\ge0}^n$, the trajectories of the system \eqref{eq:nlp} are bounded and converge to the compact set
  \begin{equation}
    \mathscr{S}:=\{x\in\mathbb{R}_{\ge0}^n:V(x)\le v^Tb/\eps\},
 \end{equation}
    i.e. $\mathscr{S}$ is forward-invariant and attractive.
\end{proposition}
\begin{proof}
  Let us consider the linear copositive Lyapunov function $V(x)=v^Tx$, $v\in\mathbb{R}_{>0}^n$. The derivative of this function along the solutions of \eqref{eq:nlp} is given by $\dot{V}(x)=v^TAx+v^TBf(x)+v^Tb$. Assuming the conditions of the theorem hold, then we have that $\dot{V}(x)\le -\eps V(x)+v^Tb$ and, therefore, that $V(x(t))\le e^{-\eps t}V(x_0)+(1-e^{-\eps t})v^Tb/\eps$. Hence, for all $x_0\notin\mathscr{S}$, the trajectories of the system \eqref{eq:nlp} converge to the compact set which proves the result.
\end{proof}
We can now derive the following structural version of the above result:
\begin{proposition}
  Let $M\in\mathbb{MS}^{n}$ and suppose further that
  \begin{enumerate}[(a)]
    \item $M$ is sign-stable,
    \item $Z:=(M^{-1}B)^T\in\mathbb{S}^{\ell\times n}$ and
    \item $Z$ is an $L^+$-matrix.
  \end{enumerate}
   Then, for all $A\in\mathcal{Q}(M)$, there exist a $v\in\mathbb{R}_{>0}^n$ and an $\eps>0$ such that $v^TA\le-\eps v^T$ and $v^TB=0$. Therefore, the trajectories of the system \eqref{eq:nlp} are structurally bounded and structurally converge to a compact set inside the nonnegative orthant.
\end{proposition}

\subsection{Ergodicity of stochastic reaction networks}\label{sec:ergo}

\blue{Reaction networks \cite{Horn:72,Feinberg:72,Anderson:15} is a powerful modeling paradigm for representing processes arising, for instance, in biology \cite{Anderson:15,Briat:13i,Briat:15e}. We consider here a stochastic reaction network denoted by $(\Xz,\mathcal{R})$ involving $d$ molecular species $\X{1},\ldots,\X{d}$ that interact through $K$ reaction channels $\mathcal{R}_1,\ldots,\mathcal{R}_K$ defined as
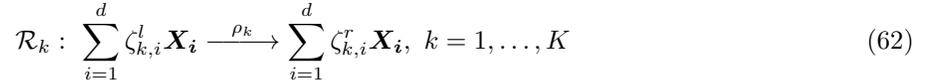
\begin{equation}
 \mathcal{R}_k:\ \sum_{i=1}^d\zeta_{k,i}^l\X{i}\rarrow{\rho_k}  \sum_{i=1}^d\zeta_{k,i}^r\X{i},\ k=1,\ldots,K
\end{equation}
where $\rho_k\in\mathbb{R}_{>0}$ is the reaction rate parameter and $\zeta_{k,i}^l,\zeta_{k,i}^r\in\mathbb{Z}_{\ge0}$. Each reaction is additionally described by a stoichiometric vector and a propensity function. The stoichiometric vector of reaction  $\mathcal{R}_k$ is given by $\zeta_k:=\zeta_k^r-\zeta_k^l\in\mathbb{Z}^d$ where $\zeta_k^r=\col(\zeta_{k,1}^r,\ldots,\zeta_{k,d}^r)$ and $\zeta_k^l=\col(\zeta_{k,1}^l,\ldots,\zeta_{k,d}^l)$. In this regard, when the reaction $\mathcal{R}_k$ fires, the state jumps from $x$ to $x+\zeta_k$. We define the stoichiometry matrix $S\in\mathbb{Z}^{d\times K}$ as $S:=\begin{bmatrix}
  \zeta_1\ldots\zeta_K
\end{bmatrix}$. When  the kinetics is  mass-action, the propensity function of reaction $\mathcal{R}_k$ is given by  $\textstyle\lambda_k(x)=\rho_k\prod_{i=1}^d\frac{x_i!}{(x_i-\zeta_{n,i}^l)!}$ and is such that  $\lambda_k(x)=0$ if $x\in\mathbb{Z}_{\ge0}^d$ and $x+\zeta_k\notin\mathbb{Z}_{\ge0}^d$. We restrict ourselves here to reaction networks involving at most bimolecular reactions. That is, the propensities functions are at most polynomials of degree two and, therefore, we can write the propensity vector as $\lambda(x)=\col(w_0,Wx,Y(x))$ where $w_0\in\mathbb{R}^{n_0}_{\ge0}$, $Wx\in\mathbb{R}^{n_u}_{\ge0}$ and $Y(x)\in\mathbb{R}^{n_b}_{\ge0}$ are the propensity vectors associated the zeroth-, first- and second-order reactions, respectively. According to this structure, the stoichiometric matrix is decomposed as
 $S=:\begin{bmatrix}S_0 &
  S_u & S_b
\end{bmatrix}$. Under the well-mixed assumption, this network can be described by a continuous-time Markov process $(X_1(t),\ldots,X_d(t))_{t\ge0}$ with state-space $\mathbb{Z}_{\ge0}^d$; see e.g. \cite{Anderson:15}. A fundamental result from \cite{Meyn:93} states that the ergodicity of an irreducible\footnote{The state-space of the Markov process is said to be irreducible if any state can be reached from any other state with positive probability. When the state-space is finite, then this reduces to the irreducibility of the transition-rate matrix. A method for proving the irreducibility of the state-space in the infinite case has been proposed in \cite{Gupta:13}.} continuous-time Markov process can be checked using a Lyapunov-like condition, referred to as a Foster-Lyapunov condition. This has led to the following result:

\begin{theorem}[\cite{Briat:13i}]
  Let us consider an irreducible\footnote{Computationally tractable conditions for checking the irreducibility of reaction networks are provided in \cite{Gupta:13}.} bimolecular reaction network $(\Xz,\mathcal{R})$ and define $A:=S_uW$. Assume that there exists a vector $v\in\mathbb{R}^d_{>0}$ such that $v^TS_b=0$ and $v^TA<0$. Then, the  reaction network $(\Xz,\mathcal{R})$ is exponentially ergodic and all the moments are bounded and converging.
\end{theorem}}

\noindent\bluee{The above result can be made structural the following way:
\begin{proposition}\label{pro:strctergo}
Let $Z\in\mathbb{MS}^{n\times n}$ and suppose, further, that
  \begin{enumerate}[(a)]
    \item the matrix $Z$ is sign-stable,
    \item the matrix $Y:=(Z^{-1}S_b)^T$ is a sign-matrix, and
    \item the matrix $Y$ is an $L^+$-matrix.
  \end{enumerate}
  Then, for all $A\in\mathcal{Q}(Z)$, there exists a $v\in\mathbb{R}_{>0}^n$ such that $v^TA<0$ and $v^TS_b=0$, which implies, in turn, that for all $S_uW=A\in\mathcal{Q}(Z)$, the reaction network $(\Xz,\mathcal{R})$ is exponentially ergodic and all the moments are bounded and converging.
\end{proposition}}

\section*{Bibliography}

\bibliographystyle{unsrtnat}

\end{document}